\documentclass[12pt]{amsart}
\usepackage[margin=1.2in]{geometry}
\usepackage{graphicx,latexsym}
\usepackage{comment}
\usepackage{tikz}
\usepackage{amsfonts, amssymb, amsmath, amsthm, bm, hyperref}

\newcommand{\N}{\mathbb{N}}

\newcommand{\R}{\mathbb{R}}
\newcommand{\C}{\mathbb{C}}

\newcommand{\dx}{{\rm d}x }

\newcommand{\dt}{{\rm d}t }
\newcommand{\dxi}{{\rm d}\xi }

\newcommand{\dl}{{\rm d}\lambda }

\newcommand{\dgamma}{{\rm d}\gamma}

\newtheorem{theorem}{Theorem}[section]
\newtheorem{proposition}[theorem]{Proposition}

\newtheorem{lemma}[theorem]{Lemma}
\newtheorem{corollary}[theorem]{Corollary}

\theoremstyle{definition}

\theoremstyle{remark}
\newtheorem{remark}[theorem]{Remark}

\numberwithin{equation}{section}

\begin{document}
\title[Weighted inductive limits of spaces of ultradifferentiable functions]{On weighted inductive limits of spaces of ultradifferentiable functions and their duals
}
\author[A. Debrouwere]{Andreas Debrouwere}
\address{Department of Mathematics, Ghent University, Krijgslaan 281, 9000 Gent, Belgium}
\email{Andreas.Debrouwere@UGent.be}
\thanks{A. Debrouwere gratefully acknowledges support by Ghent University, through a BOF Ph.D.-grant.}

\author[J. Vindas]{Jasson Vindas}
\thanks{The work of J. Vindas was supported by Ghent University, through the BOF-grants 01N01014 and 01J04017.}
\address{Department of Mathematics, Ghent University, Krijgslaan 281, 9000 Gent, Belgium}
\email{jasson.vindas@UGent.be}

\begin{abstract}
In the first part of this paper we discuss the completeness of two general classes of weighted inductive limits of  spaces of ultradifferentiable functions. In the second part we study their duals and characterize these spaces in terms of the growth of convolution averages of their elements. This characterization gives a canonical way to define a locally convex topology on these spaces and we give necessary and sufficient conditions for them to be ultrabornological. In particular, our results apply to spaces of convolutors for Gelfand-Shilov spaces. 
\end{abstract}
\maketitle

\section{Introduction}
The determination of topological properties of $(LF)$-spaces of functions is an important problem in functional analysis that usually demands a delicate treatment. In the case of weighted inductive limits of spaces of (vector-valued) continuous and holomorphic functions the subject has a long tradition that goes back to the work of Bierstedt, Meise, and Summers \cite{B-M,B-M-S}. These kinds of spaces naturally arise in numerous fields of analysis like linear partial differential equations, Fourier analysis, or analytic representations of (ultra)distribution spaces. 
Inspired by this line of research, we introduce and study in the first part of this paper two general classes of weighted inductive limits of spaces of \emph{ultradifferentiable} functions. These spaces can be viewed as natural counterparts of the space $\mathcal{O}_C(\R^d)$ of ``very slowly increasing smooth functions'' in the ultradifferentiable setting; one type corresponding to the Beurling case and the other one to the Roumieu case. 
 
The second part of the article is devoted to studying the topological duals of these spaces. Our first goal is to characterize these duals in terms of the growth of convolution averages of their elements, thereby generalizing various classical results of Schwartz \cite{Schwartz} from distributions to ultradistributions; see \cite{D-P-P-V,P-P-V} for earlier work in this direction. Schwartz (and the authors of \cite{D-P-P-V,P-P-V}) use the  parametrix method while we develop here a completely different approach based on descriptions of these ultradistribution spaces in terms of the \emph{short-time Fourier transform (STFT)}. This approach allows us to work under much milder assumptions than those needed when working with the parametrix method. In this respect, we mention the interesting paper \cite{B-O} by Bargetz and Ortner in which the mapping properties of the STFT on $\mathcal{O}'_C(\R^d)$ are established by using Schwartz' theory of vector-valued distributions. Our next aim is to study the topological properties of the duals under consideration and characterize when they are ultrabornological. This  can be regarded as the ultradistributional analogue of the last part of Grothendieck's doctoral thesis \cite{Grothendieck}. Our method however entirely differs from the one employed by Grothendieck; again, our arguments exploit the STFT.

Our results apply to the important case of spaces of convolutors for Gelfand-Shilov spaces. Such convolutor spaces have already been considered in \cite{D-P-Ve, D-P-P-V}; we shall improve various of the results shown there. For instance, we also treat here the quasianalytic case, and, moreover, Corollary \ref{topology-OCD} and Corollary \ref{topology-OCD-Roumieu} essentially solve the question posed after \cite[Thm.\ 3.3]{D-P-P-V}.  It is worth mentioning that convolutor spaces appear naturally in the study of abstract partial differential and convolution equations, see e.g.\ \cite{Chazarain, C-Z,Kostic, S-Z}.

The plan of the article is as follows. In the preliminary Section \ref{sect-prel} we recall some notions from the theory of $(LF)$-spaces that will be frequently used throughout the first part of the paper. We also discuss there the mapping properties of the STFT on Gelfand-Shilov spaces and their duals. In Section \ref{sect-reg-ultra}  we introduce two new types of weighted inductive limits of ultradifferentiable functions and characterize when these spaces are complete in terms of the defining family of weight functions. Our arguments make use of a result of Albanese \cite{Albanese} on the completeness of weighted inductive limits of spaces of Fr\'echet-valued continuous functions (which will be discussed in Subsection \ref{sect-reg-cont}). Section \ref{sect-duals} deals with the duals of our weighted inductive limits of spaces of ultradifferentiable functions.  As mentioned before, a key to our arguments is to employ the STFT. Subsection \ref{Char STFT dual inductive} provides a detailed study of the mapping properties of the STFT on our spaces. We establish the sought convolution average characterizations in Subsection \ref{conv average subsection}. The characterization in terms of convolution averages suggests a natural way to define a locally convex topology on these duals and, based upon the results from Section \ref{sect-reg-ultra} and the mapping properties of the STFT, we give in Subsection \ref{subsection topological properties} necessary and sufficient conditions for these spaces to be ultrabornological.

\section{preliminaries}\label{sect-prel}
In the first part of this preliminary section we recall several regularity conditions for $(LF)$-spaces and state how they are related to each other; see \cite{Wengenroth-96} and \cite[Chap.\ 6]{Wengenroth} for more detailed accounts on the subject. Secondly, we discuss a result of Albanese \cite{Albanese} concerning the completeness of weighted inductive limits of spaces of Fr\'echet-valued continuous functions. This result will play a key role in Section \ref{sect-reg-ultra}. Next, we collect some facts about the Gelfand-Shilov spaces $\mathcal{S}^{(M_p)}_{(A_p)}(\R^d)$ and $\mathcal{S}^{\{M_p\}}_{\{A_p\}}(\R^d)$  and their duals. We also discuss the mapping properties of the short-time Fourier transform on these spaces, they will be employed in Section \ref{sect-duals}. 
%Throughout this article every locally convex space (from now on abbreviated as l.c.s.) is assumed to be Hausdorff.

\subsection{Regularity conditions for $(LF)$-spaces}\label{sect-reg-cond}  
%A  l.c.s.\ 
A Hausdorff l.c.s. (locally convex space) $E$ is called an $(LF)$-space if there is a sequence $(E_n)_{n \in \N}$ of Fr\'echet spaces with $E_n \subset E_{n + 1}$ and continuous inclusion mappings such that $E = \bigcup_{n \in \N} E_n$ and the topology of $E$ coincides with the finest locally convex topology for which all inclusion mappings $E_n \rightarrow E$ are continuous. We call  $(E_n)_{n}$ a defining inductive spectrum for $E$ and write $E = \varinjlim E_n$. We emphasize that, for us, $(LF)$-spaces are Hausdorff by definition. If the sequence $(E_n)_{n}$ consists of $(FS)$-spaces, $E$ is called an $(LFS)$-space. Similarly,  $E$ is said to be an $(LB)$-space if  the sequence $(E_n)_{n}$ consists of Banach spaces.

An $(LF)$-space $E=\varinjlim E_n$ is said to be \emph{regular} if every bounded set $B$ in $E$ is contained and bounded in $E_n$ for some $n \in \N$. By Grothendieck's factorization theorem (see e.g.\ \cite[p.\ 225 (4)]{Kothe}), every quasi-complete $(LF)$-space is regular.  We now discuss several related concepts. An $(LF)$-space $E = \varinjlim E_n$  is said to satisfy  condition $(wQ)$ if for every $n \in \N$ there are a neighborhood $U$ of $0$ in $E_n$ and $m > n$ such that for every $k > m$ and every neighborhood $W$ of $0$ in $E_m$ there are neighborhood $V$ of $0$ in $E_k$ and $C > 0$ with $V \cap U \subseteq CW$.  If $ (\| \, \cdot \,  \|_{n,N})_{N \in \N}$ is a fundamental sequence of seminorms for $E_n$, then $E$ satisfies $(wQ)$ if and only if
\begin{gather*}
\forall n \in \N \, \exists m > n \, \exists N \in \N \, \forall k > m \, \forall M \in \N \, \exists K \in \N \, \exists C > 0\, \forall e \in E_n: \\
\|e\|_{m,M} \leq C(\|e\|_{n,N} + \|e\|_{k,K}).
\end{gather*}
Clearly, every $(LB)$-space satisfies $(wQ)$. Moreover, every regular $(LF)$-space satisfies $(wQ)$ 
\cite[Thm.\ 4.7]{Vogt-92}.  Next, we introduce two strong regularity conditions. An $(LF)$-space $E$ is said to be \emph{boundedly retractive} if for every bounded set $B$ in $E$ there is $n \in \N$ such that $B$ is contained in $E_n$ and $E$ and $E_n$ induce the same topology on $B$, while
$E$ is said to be \emph{sequentially retractive} if for every null sequence in $E$ there is $n \in \N$ such that the sequence is contained and converges to zero in $E_n$. Finally, $E$ is said to be \emph{boundedly stable} if for every $n \in \N$ and every bounded set $B$ in $E_n$ there is $m \geq n$ such that for every $k \geq m$ the spaces $E_m$ and $E_k$ induce the same topology on $B$.  

Notice that, in view of Grothendieck's factorization theorem, these conditions do not depend on the defining inductive spectrum of $E$. This justifies calling an $(LF)$-space boundedly retractive, etc., if one (and thus all) of its defining inductive spectra has this property. These concepts are related to each other in the following way:
\begin{theorem}[{\cite[Thm.\ 6.4 and Cor.\ 6.5]{Wengenroth}}] \label{reg-cond} Let $E$ be an $(LF)$-space. The following statements are equivalent:
\begin{itemize}
\item[$(i)$] $E$ is boundedly retractive.
\item[$(ii)$] $E$ is sequentially retractive.
\item[$(iii)$] $E$ is boundedly stable and satisfies $(wQ)$.
\end{itemize}
In such a case, $E$ is complete.
\end{theorem}
In fact, the conditions in Theorem \ref{reg-cond} are also equivalent to the fact that $E$ is acyclic or that $E$ satisfies Retakh's condition $(M)$. 
We refer to \cite{Wengenroth-96} and the discussion above \cite[Thm.\ 6.4, p.\ 112]{Wengenroth} for  the history of this important theorem. 
\subsection{Weighted inductive limits of spaces of Fr\'echet-valued continuous functions}\label{sect-reg-cont}
Let $X$ be a completely regular Hausdorff topological space. A (pointwise) decreasing sequence $\mathcal{V} := (v_n)_{n \in \N}$ of positive continuous functions on $X$ is called a \emph{decreasing weight system on X}. Given a Fr\'echet space $E$ with a fundamental sequence of seminorms $(\| \, \cdot \, \|_N)_{N}$ we define $Cv_n(X;E)$ as the Fr\'echet space consisting of all $f \in C(X;E)$ such that
$$
\sup_{x \in X} \|f(x)\|_N v_n(x)  < \infty
$$
for all $N \in \N$. We set
$$
\mathcal{V}C(X;E) := \varinjlim_{n \in \N} Cv_n(X;E),
$$
an $(LF)$-space (it is Hausdorff because the topology of $\mathcal{V}C(X;E)$ is finer than the one induced by $C(X;E)$ and the latter space is Hausdorff). If $E = \C$, we simply write $\mathcal{V}C(X;\C) = \mathcal{V}C(X)$, an $(LB)$-space. We remark that $\mathcal{V}C(X)$ is always complete \cite{B-B} while it is boundedly retractive if and only if $\mathcal{V}$ is \emph{regularly decreasing} \cite[p.\ 118 Thm.\ 7]{Bierstedt}, i.e., for every $n \in \N$ there is $m \geq n$ such that for every $k > m$ and every subset $Y$ of $X$ it holds that
$$
\inf_{y \in Y} \frac{v_m(y)}{v_n(y)} > 0 \Longrightarrow \inf_{y \in Y} \frac{v_k(y)}{v_n(y)}  > 0.
$$
For example, constant weight systems and weight systems $\mathcal{V} = (v_n)_{n}$ satisfying
\begin{equation}
\forall n \in \N \, \exists m > n \, : \,  \mbox{$v_m/v_n$ vanishes at $\infty$}
\label{decay-weights-1}
\end{equation}
are regularly decreasing.
 Albanese characterized the completeness of the $(LF)$-space $\mathcal{V}C(X;E)$ in the ensuing way: 
\begin{theorem}[{\cite[Thm.\ 2.3]{Albanese}}]\label{completeness-F}
Let $\mathcal{V} = (v_n)_{n}$ be a decreasing weight system. For a non-normable Fr\'echet space $E$ with a fundamental increasing sequence of seminorms $(\| \, \cdot \,  \|_N)_{N}$ the following statements are equivalent:
\begin{itemize}
\item[$(i)$] $\mathcal{V}C(X;E)$ is boundedly retractive.
\item[$(ii)$] $\mathcal{V}C(X;E)$ is (quasi-)complete.
\item[$(iii)$] $\mathcal{V}C(X;E)$ is regular.
\item[$(iv)$] $\mathcal{V}C(X;E)$ satisfies $(wQ)$.
\item[$(v)$] The pair $(E, \mathcal{V})$ satisfies $(S_2)^*$, i.e.,
\begin{gather*}
\forall n \in \N \, \exists m > n \, \exists N \in \N \, \forall k > m \, \forall M \in \N \, \exists K \in \N \, \exists C > 0\, \forall e \in E \,  \forall x \in X: \\
v_m(x)\|e\|_M \leq C(v_n(x)\|e\|_N + v_k(x)\|e\|_K).
\end{gather*}
\end{itemize}
\end{theorem}
\begin{remark}\label{S-2-inv}
Let $\mathcal{V}$ be a decreasing weight system. If $E$ and $F$ are topologically isomorphic Fr\'echet spaces, then $(E, \mathcal{V})$ satisfies $(S_2)^*$ if and only if  $(F, \mathcal{V})$ does so.
\end{remark}
We now state separate conditions on $E$ and $\mathcal{V}$ which ensure that the pair $(E, \mathcal{V})$ satisfies $(S_2)^*$. Since this is very similar to the analysis of the conditions $(S^*_1)$ and $(S^*_2)$ in the splitting theory of Fr\'echet spaces \cite{Vogt-87}, we omit all proofs.

A Fr\'echet space $E$ with a fundamental increasing sequence of seminorms $(\| \, \cdot \,  \|_N)_{N}$ is said to satisfy $(DN)$ if
$$
\exists N \in \N \, \forall M > N \, \exists K > M \, \, \exists C > 0 \, \forall e \in E: \, \|e\|^2_M \leq C \|e\|_N\|e\|_K,
$$
while it is said to satisfy $(\Omega)$ if 
$$
\forall N \in \N \, \exists M > N \, \forall K > M \, \exists \theta \in (0,1) \, \exists C > 0 \, \forall e' \in E': \, \|e'\|^*_M \leq  C (\|e'\|^*_N)^{1-\theta}( \|e'\|^*_K)^{\theta},
$$

where
$$
\|e'\|^*_N := \sup \{ |\langle e', e \rangle | \, : \, e \in E, \,  \|e\|_N \leq 1 \} \in [0,\infty], \qquad e' \in E'.
$$

A decreasing weight system $\mathcal{V} = (v_n)_{n}$ is said to satisfy $(\Omega)$ if 
$$
\forall n \in \N \, \exists m > n \, \forall k > m \, \exists \theta \in (0,1) \, \exists C > 0 \, \forall x \in X:  \, v_m(x) \leq C v_n(x)^{1-\theta}v_k(x)^{\theta}.
$$
This terminology is justified because one  can show that $\mathcal{V}$ satisfies $(\Omega)$ if and only if the strong dual of $\mathcal{V}C(X)$ satisfies $(\Omega)$; indeed, by using an obvious analogue of \cite[Lemma 20]{B-D}, this follows from a similar argument as in \cite[Lemma 2.1]{Albanese}. We then have:
\begin{proposition} [{cf.\ \cite[Thm.\ 5.1]{Vogt-87}}]\label{(DN)-Omega-1}
Let $\mathcal{V}$ be a decreasing weight system and let $E$ be a Fr\'echet space. If $\mathcal{V}$ satisfies $(\Omega)$ and $E$ satisfies $(DN)$, then $(E,\mathcal{V})$ satisfies $(S_2)^\ast$.
\end{proposition}
If $E$ is a power series space, the conditions in Proposition \ref{(DN)-Omega-1} turn out to be necessary as well. In the rest of this subsection $\beta = (\beta_j)_{j \in \N}$ will stand for a strictly increasing sequence of positive numbers such that $\beta_j  \rightarrow \infty$.  We define $\Lambda_\infty(\beta)$ ($\Lambda_0(\beta)$, respectively) as the Fr\'echet space consisting of all sequences $(c_j)_j \in \C^\N$ such that
$$
\left(\sum_{j = 0}^\infty |c_j|^2e^{2n\beta_j}\right)^{1/2} < \infty \qquad \left(\left(\sum_{j = 0}^\infty |c_j|^2e^{-2\beta_j/n}\right)^{1/2} <\infty \right)
$$
for all $n \in \N$. Furthermore, we shall always assume that $\beta$ is shift-stable, i.e.,
\begin{equation}
\sup_{j \in \N} \frac{\beta_{j+1}}{\beta_j} < \infty.
\label{stable}
\end{equation}
\begin{proposition}[{cf.\ \cite[Thm.\ 4.1]{Vogt-87}}]\label{E-fixed-1}
Let $\mathcal{V}$ be a decreasing weight system. Then, $\mathcal{V}$ satisfies $(\Omega)$ if and only if $(\Lambda_\infty(\beta),\mathcal{V})$ satisfies $(S_2)^\ast$.
\end{proposition}

Finally, we introduce increasing weight systems. They will also play an essential role in the rest of this article. An increasing sequence $\mathcal{W} := (w_N)_{N \in \N}$ of positive continuous functions on $X$ is called an \emph{increasing weight system on X}. Sometimes we shall impose the following condition on $\mathcal{W}$ (cf.\ \eqref{decay-weights-1}):
\begin{equation}
\forall N \in \N \, \exists M > N \, : \mbox{$w_N/w_M$ vanishes at $\infty$}.
\label{decay-weights}
\end{equation}
We define $\mathcal{W}C(X)$ as the Fr\'echet space consisting of all $f \in C(X)$ such that
$$
\| f \|_{w_N} := \sup_{x \in X}|f(x)|w_N(x)  < \infty
$$
for all $N \in \N$. Let $E = \varinjlim E_n$ be an $(LB)$-space. The pair $(\mathcal{W}, E)$ is said to satisfy $(S_2)^*$ if
\begin{gather*}
\forall n \in \N \, \exists m > n \, \exists N \in \N \, \forall k > m \, \forall M \in \N \, \exists K \in \N \, \exists C > 0\, \forall e \in E_n \,  \forall x \in X: \\
\|e\|_{E_m} w_M(x) \leq C(\|e\|_{E_n}w_N(x) + \|e\|_{E_k}w_K(x)).
\end{gather*}
\begin{remark}\label{S-2-inv-1}
Let $\mathcal{W}$ be an increasing weight system. If $E$ and $F$ are topologically isomorphic $(LB)$-spaces, then $(\mathcal{W},E)$ satisfies $(S_2)^*$ if and only if  $(\mathcal{W},F)$ does so, as follows from Grothendieck's factorization theorem. In particular, condition $(S_2)^*$ does not depend on the defining inductive spectrum of $E$.
\end{remark}
We now wish to formulate an analogue of Proposition \ref{E-fixed-1} for increasing weight systems.
The weight system $\mathcal{W} = (w_N)_N$ is said to satisfy $(DN)$ if 
$$
\exists N \in \N \, \forall M > N \, \exists K > M \,  \exists C > 0 \, \forall x \in X: \, w^2_M(x) \leq C w_N(x)w_K(x).
$$
As before, this terminology is justified because one can show that  $\mathcal{W}$ satisfies $(DN)$ if and only if $\mathcal{W}C(X)$ satisfies $(DN)$. 

\begin{proposition}[{cf.\ \cite[Thm.\ 4.3]{Vogt-87}}]\label{W-fixed-1}
Let $\mathcal{W}$ be an increasing weight system. Then, $\mathcal{W}$ satisfies $(DN)$ if and in only if $(\mathcal{W}, \Lambda'_0(\beta))$ satisfies $(S_2)^\ast$.
\end{proposition}

\subsection{Gelfand-Shilov spaces and the short-time Fourier transform.}\label{gelfand-shilov-spaces} 
Let $(M_p)_{p \in \N}$ be a sequence of positive real numbers and define $m_p := M_p / M_{p -1}$, $p \geq 1$. We call $M_p$ a \emph{weight sequence} if $\lim_{p \to \infty} m_p = \infty$. Furthermore, we will make use of some of the following conditions:
\begin{enumerate}
\item[$ (M.1)$ ] $M_p^2 \leq M_{p-1}M_{p+1}$, $p \geq 1$.
\item[$(M.2)$ ] $M_{p+q} \leq C_0H^{p+q}M_pM_q$, $p,q \in \N$, for some $C_0,H \geq 1$.
\item[ $(M.2)^\ast $] $2m_p \leq m_{Qp}$, $p \geq 1$, for some $Q \in \N$.
\item[$(M.3)$ ] $\sum_{p = q}^\infty 1/m_p< Cq/m_{q}$, $q  \geq 1$, for some $C \geq 1$.
\end{enumerate}
The reader can consult  \cite{Komatsu,B-M-M} for the meaning of these conditions. It is worth mentioning that $(M.1)$  and $(M.3)$ imply $(M.2)^*$ \cite[Prop.\ 1.1]{Petzsche88}. We write $M_\alpha = M_{|\alpha|}$, $\alpha \in \N^d$.  As usual, the relation $M_p\subset N_p$ between two weight sequences means that there are $C,h>0$ such that 
$M_p\leq Ch^{p}N_{p},$ $p\in\N$.  The stronger relation $M_p\prec N_p$ means that the latter inequality remains valid for every $h>0$ and a suitable $C=C_{h}>0$. 

The \emph{associated function of $M_p$} is defined as
$$
M(t):=\sup_{p\in\N}\log\frac{t^pM_0}{M_p},\qquad t > 0,
$$
and $M(0):=0$. We define $M$ on $\R^d$ as the radial function $M(x) = M(|x|)$, $x \in \R^d$. Under $(M.1)$, the assumption $(M.2)$ holds \cite[Prop.\ 3.6]{Komatsu} if and only if
\begin{equation*}
2M(t) \leq M(Ht) + \log C_0, \qquad  t \geq 0,
%\label{M.2-ass-function}
\end{equation*}
while, under $(M.1)$ and $(M.2)$, the assumption $(M.2)^*$ \cite[Prop.\ 13]{B-M-M} holds if and only if
\begin{equation*}
M(2t) \leq H'M(t) + \log C_0', \qquad t \geq 0,
%\label{M.2*-associated-function}
\end{equation*}
for some $C_0',H' \geq  1$. 
Let $M_p$ and $A_p$ be two weight sequences. We denote by $A$ the associated function of $A_p$. For $h, \lambda > 0$ we write $\mathcal{S}^{M_p,h}_{A_p,\lambda}(\R^d)$ for the Banach space consisting of all $\varphi \in C^\infty(\R^d)$ such that
$$
\| \varphi \|_{\mathcal{S}^{M_p,h}_{A_p,\lambda}} :=\sup_{\alpha \in \N^d} \sup_{x \in \R^d} \frac{h^{|\alpha|}|\varphi^{(\alpha)}(x)|e^{A(\lambda x)}}{M_{\alpha}} < \infty. 
$$
We define
$$
\mathcal{S}^{(M_p)}_{(A_p)}(\R^d) := \varprojlim_{h \to \infty} \mathcal{S}^{M_p,h}_{A_p,h}(\R^d), \qquad \mathcal{S}^{\{M_p\}}_{\{A_p\}}(\R^d) := \varinjlim_{h \to 0^+} \mathcal{S}^{M_p,h}_{A_p,h}(\R^d).
$$
Elements of their dual spaces $\mathcal{S}'^{(M_p)}_{(A_p)}(\R^d)$ and $\mathcal{S}'^{\{M_p\}}_{\{A_p\}}(\R^d) $ are called \emph{tempered ultradistributions of Beurling and Roumieu type}, respectively. In the sequel we shall write $\ast$ instead of $(M_p)$ or $\{M_p\}$ and $\dagger$ instead of $(A_p)$ or $\{A_p\}$ if we want to treat both cases simultaneously. In addition, we shall often first state assertions for the Beurling case followed in parenthesis by the corresponding statements for the Roumieu case.

Let $M_p$ and $A_p$ be weight sequences. We introduce the following set of conditions on $M_p$ and $A_p$: 
\begin{equation}
\mbox{$M_p$ and $A_p$ satisfy $(M.1)$ and $(M.2)$, $p! \prec M_pA_p$, and $\mathcal{S}^{(M_p)}_{(A_p)}(\R^d)$ is non-trivial.}
\label{group-cond}
\end{equation}
 A sufficient condition for the non-triviality of $\mathcal{S}^{(M_p)}_{(A_p)}(\R^d)$ is $p!^\sigma \subset M_p$ and $p!^\tau \subset A_p$ for some $\sigma, \tau > 0$ with $\sigma + \tau > 1$ \cite[p.\ 235]{G-S}. Other non-triviality conditions can be found in \cite{Debrouwere-VindasUH2016}. Under the general conditions
 \eqref{group-cond}, we have the ensuing properties:
\begin{itemize}
\item[$(i)$] The Fourier transform is a topological isomorphism from $\mathcal{S}^\ast_\dagger(\R^d)$  onto $\mathcal{S}^\dagger_\ast(\R^d)$, where we fix the constants in the Fourier transform as follows
$$
\mathcal{F}(\varphi)(\xi)  =\widehat{\varphi}(\xi) := \int_{\R^d}\varphi(x) e^{-2\pi i x\xi}\dx.
$$
We define the Fourier transform from $\mathcal{S}'^\dagger_\ast(\R^d)$ onto $\mathcal{S}'^\ast_\dagger(\R^d)$ via duality.
\item[$(ii)$]  $\mathcal{S}^{(M_p)}_{(A_p)}(\R^d)\hookrightarrow \mathcal{S}^{\{M_p\}}_{\{A_p\}}(\R^d)  $ (the symbol ``$\hookrightarrow$'' stands for dense and continuous inclusion) \cite[Lemma 2.4]{P-P-V}. 
\item[$(iii)$] $\mathcal{S}^{(M_p)}_{(A_p)}(\R^d)$ is an $(FN)$-space while $\mathcal{S}^{\{M_p\}}_{\{A_p\}}(\R^d)$ is a $(DFN)$-space \cite[Prop.\ 2.11]{P-P-V}.
\end{itemize}
We shall use these properties without explicitly referring to them.

Next, we discuss the mapping properties of the short-time Fourier transform on the spaces $\mathcal{S}^\ast_\dagger(\R^d)$ and $\mathcal{S}'^\ast_\dagger(\R^d)$. The translation and modulation operators are denoted by $T_xf = f(\:\cdot\: - x)$ and $M_\xi f = e^{2\pi i \xi \cdot} f$, for 
$x, \xi \in \R^d$. We also write $\check{f} = f(-\ \cdot\ )$ for reflection about the origin. The short-time Fourier transform (STFT) of a function $f \in L^2(\R^d)$ with respect to a window function $\psi \in L^2(\R^d)$ is defined as
$$
V_\psi f(x,\xi) := (f, M_\xi T_x\psi)_{L^2} = 
\int_{\R^d} f(t) \overline{\psi(t-x)}e^{-2\pi i \xi t} \dt, \qquad (x, \xi) \in \R^{2d}.
$$
It holds that $\|V_\psi f\|_{L^2(\R^{2d})} = \|\psi \|_{L^2}\|f\|_{L^2}$. In particular, the linear mapping $V_\psi : L^2(\R^d) \rightarrow L^2(\R^{2d})$ is continuous. The adjoint of $V_\psi$ is given by the weak integral
$$
V^\ast_\psi F = \int \int_{\R^{2d}} F(x,\xi) M_\xi T_x\psi \dx \dxi, \qquad F \in L^2(\R^{2d}).
$$
If $\psi \neq 0$ and $\gamma \in L^2(\R^d)$ is a synthesis window for $\psi$, that is, $(\gamma, \psi)_{L^2} \neq 0$, then
\begin{equation}
\frac{1}{(\gamma, \psi)_{L^2}} V^\ast_\gamma \circ V_\psi = \operatorname{id}_{L^2(\R^d)}.
\label{reconstruction-L2}
\end{equation}
We refer to \cite{Grochenig} for further properties of the STFT. Throughout  the rest of this subsection $M_p$ and $A_p$ will stand for weight sequences satisfying \eqref{group-cond}. As usual, given two l.c.s.\ $E$ and $F$ we write $E\widehat{\otimes}F$ for the completion of the tensor product $E\otimes F$ with respect to the $\varepsilon$- or $\pi$-topology provided that either $E$  or $F$ is nuclear.

Let $N_p$ and $B_p$ be two other weight sequences satisfying \eqref{group-cond}. We denote by $B$ the associated function of $B_p$. For $h>0$ we write $X_h$ for the Banach space consisting of all $\varphi \in C^\infty(\R^{d_1+d_2})$ such that 
$$
\sup_{(\alpha, \beta) \in \N^{d_1+d_2}} \sup_{(x, t) \in \R^{d_1+d_2}} \frac{h^{|\alpha| + |\beta|}|\partial^\beta_t \partial^\alpha_x\varphi(x,t)|e^{A(hx) + B(ht)}}{M_\alpha N_\beta} < \infty;
$$
one can then show that (cf.\ \cite[Prop.\  2.12]{P-P-V})
$$
 \mathcal{S}^{(M_p)}_{(A_p)}(\R^{d_1}_x) \widehat{\otimes} \mathcal{S}^{(N_p)}_{(B_p)}(\R^{d_2}_t) \cong \varprojlim_{h \to \infty} X_h, \qquad  \mathcal{S}^{\{M_p\}}_{\{A_p\}}(\R^{d_1}_x) \widehat{\otimes} \mathcal{S}^{\{N_p\}}_{\{B_p\}}(\R^{d_2}_t) \cong \varinjlim_{h\to 0^{+}} X_{h},
$$
as locally convex spaces. In particular, we have that $ \mathcal{S}^\ast_\dagger(\R^{d_1}) \widehat{\otimes} \mathcal{S}^\ast_\dagger(\R^{d_2})  \cong \mathcal{S}^\ast_\dagger(\R^{d_1+ d_2})$. Furthermore, we have the following isomorphisms of l.c.s.
$$
( \mathcal{S}^{(M_p)}_{(A_p)}(\R^{d_1}_x) \widehat{\otimes} \mathcal{S}^{(N_p)}_{(B_p)}(\R^{d_2}_t))' \cong  \mathcal{S}'^{(M_p)}_{(A_p)}(\R^{d_1}_x) \widehat{\otimes} \mathcal{S}'^{(N_p)}_{(B_p)}(\R^{d_2}_t)
$$
and
$$
( \mathcal{S}^{\{M_p\}}_{\{A_p\}}(\R^{d_1}_x) \widehat{\otimes} \mathcal{S}^{\{N_p\}}_{\{B_p\}}(\R^{d_2}_t))' \cong  \mathcal{S}'^{\{M_p\}}_{\{A_p\}}(\R^{d_1}_x) \widehat{\otimes} \mathcal{S}'^{\{N_p\}}_{\{B_p\}}(\R^{d_2}_t).
$$
Naturally, the partial Fourier transforms
$$
\mathcal{F}_t :  \mathcal{S}^{(M_p)}_{(A_p)}(\R^{d_1}_x) \widehat{\otimes} \mathcal{S}^{(N_p)}_{(B_p)}(\R^{d_2}_t) \rightarrow \mathcal{S}^{(M_p)}_{(A_p)}(\R^{d_1}_x) \widehat{\otimes} \mathcal{S}^{(B_p)}_{(N_p)}(\R^{d_2}_\xi) 
$$
and
$$
\mathcal{F}_t :  \mathcal{S}^{\{M_p\}}_{\{A_p\}}(\R^{d_1}_x) \widehat{\otimes} \mathcal{S}^{\{N_p\}}_{\{B_p\}}(\R^{d_2}_t) \rightarrow \mathcal{S}^{\{M_p\}}_{\{A_p\}}(\R^{d_1}_x) \widehat{\otimes} \mathcal{S}^{\{B_p\}}_{\{N_p\}}(\R^{d_2}_\xi),
$$
given by
$$
\mathcal{F}_t(\varphi)(x, \xi) :=  \int_{\R^{d_2}}\varphi(x,t) e^{-2\pi i t\xi}\dt,
$$
are topological isomorphisms. 
\begin{proposition}\label{STFT-testfunctions} Let $\psi \in \mathcal{S}^{(M_p)}_{(A_p)}(\R^d)$. The following mappings are continuous: 
$$
V_\psi: \mathcal{S}^\ast_\dagger(\R^d) \rightarrow  \mathcal{S}^\ast_\dagger(\R^d_x) \widehat{\otimes} \mathcal{S}^\dagger_\ast(\R^d_\xi)
$$
and
$$ 
V^\ast_\psi: \mathcal{S}^\ast_\dagger(\R^d_x) \widehat{\otimes} \mathcal{S}^\dagger_\ast(\R^d_\xi)  \rightarrow  \mathcal{S}^\ast_\dagger(\R^d).
$$
\end{proposition}
\begin{proof}Consider the continuous linear mappings
$$
\cdot \otimes \overline{\psi} :\mathcal{S}^\ast_\dagger(\R^d_t) \rightarrow\mathcal{S}^\ast_\dagger(\R^{2d}_{t,y}): \varphi(t) \rightarrow \varphi(t)\otimes \overline{\psi}(y)
$$
and
$$
T: \mathcal{S}^\ast_\dagger(\R^{2d}_{t,y}) \rightarrow \mathcal{S}^\ast_\dagger(\R^{2d}_{t,x}): \chi(t,y) \rightarrow \chi(t,t -x).
$$
That $V_\psi$ is continuous then follows from the representation $V_\psi = \mathcal{F}_t \circ T \circ (\cdot \otimes \overline{\psi})$. The continuity of $V^{\ast}_{\psi}$ is an immediate consequence of (the proof of) Lemma \ref{double-int-test} shown below. 

\end{proof}
The STFT of an ultradistribution $f \in \mathcal{S}'^\ast_\dagger(\R^d)$ with respect to a window function $\psi \in \mathcal{S}^\ast_\dagger(\R^d)$ is defined as
\begin{equation*}
V_\psi f(x,\xi):= \langle f, \overline{M_\xi T_x\psi}\rangle= e^{-2\pi i \xi x}(f \ast M_\xi  \check{\overline {\psi}})(x), \qquad (x, \xi) \in \R^{2d}.
\end{equation*}
Clearly, $V_\psi f$ is a smooth function on $\R^{2d}$. We \emph{define} the adjoint STFT of $F \in \mathcal{S}'^\ast_\dagger(\R^d) \widehat{\otimes} \mathcal{S}'^\dagger_\ast(\R^d)$ as
$$
\langle V^\ast_\psi F, \varphi \rangle := \langle F, \overline{V_\psi\overline{\varphi}} \rangle, \qquad \varphi \in \mathcal{S}^\ast_\dagger(\R^d).
$$
Proposition \ref{STFT-testfunctions}, the reconstruction formula (\ref{reconstruction-L2}), and the same argument given in \cite[Sect.\ 3]{K-P-S-V2016} yield:
\begin{proposition}\label{STFT-duals} Let $\psi \in \mathcal{S}^{(M_p)}_{(A_p)}(\R^d)$. The mappings 
$$
V_\psi: \mathcal{S}'^\ast_\dagger(\R^d) \rightarrow  \mathcal{S}'^\ast_\dagger(\R^d_x) \widehat{\otimes} \mathcal{S}'^\dagger_\ast(\R^d_\xi)
$$
and
$$
V^\ast_\psi: \mathcal{S}'^\ast_\dagger(\R^d_x) \widehat{\otimes} \mathcal{S}'^\dagger_\ast(\R^d_\xi)  \rightarrow  \mathcal{S}'^\ast_\dagger(\R^d)
$$
are continuous.
Moreover, if $\gamma \in \mathcal{S}^{(M_p)}_{(A_p)}(\R^d)$ is a synthesis window for $\psi$, then
$$
\frac{1}{(\gamma, \psi)_{L^2}} V^\ast_\gamma \circ V_\psi = \operatorname{id}_{ \mathcal{S}'^\ast_\dagger(\R^d)}
$$
and the desingularization formula
\begin{equation}
\langle f, \varphi \rangle = \frac{1}{(\gamma, \psi)_{L^2}} \int \int_{\R^{2d}}V_\psi f(x, \xi) V_{\overline{\gamma}}\varphi(x, - \xi) \dx \dxi
\label{regularization-via-STFT}
\end{equation}
holds for all $f \in  \mathcal{S}'^\ast_\dagger(\R^d)$ and $\varphi \in  \mathcal{S}^\ast_\dagger(\R^d)$.
\end{proposition}
\section{The Gelfand-Shilov type spaces $\mathcal{B}^*_\mathcal{W}(\R^d)$ and $\mathcal{B}^*_\mathcal{V}(\R^d)$}\label{sect-reg-ultra}
We now introduce new classes of Gelfand-Shilov type spaces as weighted spaces of ultradifferentiable functions. Let $w$ be a nonnegative function on $\R^d$ and let $M_p$ be a weight sequence. For $h > 0$ we write $\mathcal{B}^{M_p,h}_w(\R^d)$ for the seminormed space consisting of all $\varphi \in C^\infty(\R^d)$ such that
$$
\| \varphi \|_{\mathcal{B}^{M_p,h}_w} := \sup_{\alpha \in \N^d} \sup_{x \in \R^d} \frac{h^{|\alpha|}|\varphi^{(\alpha)}(x)|w(x)}{M_\alpha} < \infty.
$$ 
If $w$ is positive and $w^{-1}$ is locally bounded, then $\mathcal{B}^{M_p,h}_w(\R^d)$ is a Banach space. These requirements are fulfilled if $w$ is positive and continuous. We set
$$
\mathcal{B}^{(M_p)}_w(\R^d) := \varprojlim_{h \rightarrow \infty} \mathcal{B}^{M_p,h}_w(\R^d),  \qquad \mathcal{B}^{\{M_p\}}_w(\R^d) := \varinjlim_{h \rightarrow 0^+} \mathcal{B}^{M_p,h}_w(\R^d).
$$
Let $\mathcal{W} = (w_N)_{N}$ be an increasing weight system on $\R^d$. We define
$$
\mathcal{B}^{M_p,h}_\mathcal{W}(\R^d) := \varprojlim_{N \in \N} \mathcal{B}^{M_p,h}_{w_N}(\R^d)
$$
and 
$$
\mathcal{B}^{(M_p)}_\mathcal{W}(\R^d) := \varprojlim_{h \rightarrow \infty} \mathcal{B}^{M_p,h}_\mathcal{W}(\R^d), \qquad \mathcal{B}^{\{M_p\}}_\mathcal{W}(\R^d) :=  \varinjlim_{h \rightarrow 0^+} \mathcal{B}^{M_p,h}_\mathcal{W}(\R^d).
$$
The spaces $\mathcal{B}^{M_p,h}_\mathcal{W}(\R^d)$ and $\mathcal{B}^{(M_p)}_\mathcal{W}(\R^d)$ are Fr\'echet spaces while $\mathcal{B}^{\{M_p\}}_\mathcal{W}(\R^d)$ is an $(LF)$-space. Similarly, given a decreasing weight system $\mathcal{V} = (v_n)_{n}$ on $\R^d$  we define  
$$
\mathcal{B}^{(M_p)}_\mathcal{V}(\R^d) := \varinjlim_{n \in \N} \mathcal{B}^{(M_p)}_{v_n}(\R^d), \qquad \mathcal{B}^{\{M_p\}}_\mathcal{V}(\R^d) :=  \varinjlim_{n \in \N} \mathcal{B}^{\{M_p\}}_{v_n}(\R^d).
$$
The space $\mathcal{B}^{(M_p)}_\mathcal{V}(\R^d)$ is an $(LF)$-space while $\mathcal{B}^{\{M_p\}}_\mathcal{V}(\R^d)$  is an $(LB)$-space. The main goal of this section is to characterize the regularity properties of the $(LF)$-spaces $\mathcal{B}^{\{M_p\}}_\mathcal{W}(\R^d)$  and $\mathcal{B}^{(M_p)}_\mathcal{V}(\R^d)$ in terms of the weight systems $\mathcal{W}$ and $\mathcal{V}$, respectively. As a first step, we discuss when these spaces are boundedly stable.

\begin{proposition} \label{boundedly-stable}
Let $M_p$ be a weight sequence and let  $\mathcal{W} = (w_N)_{N}$ be an increasing weight system. Then, $\mathcal{B}^{\{M_p\}}_\mathcal{W}(\R^d)$ is boundedly stable.
\end{proposition}
\begin{proof}
Let $h > 0$ be arbitrary and let $B$ be a bounded set of $\mathcal{B}^{M_p,h}_\mathcal{W}(\R^d)$. We shall show that, for all $0 < l \leq k < h$, the spaces $\mathcal{B}^{M_p,k}_\mathcal{W}(\R^d)$ and $\mathcal{B}^{M_p,l}_\mathcal{W}(\R^d)$ induce the same topology on $B$. Clearly, it is enough to prove that the filter of neighborhoods of each point of $B$ induced by $\mathcal{B}^{M_p,l}_\mathcal{W}(\mathbb{R}^d)$ is finer than the one induced by $\mathcal{B}^{M_p,k}_\mathcal{W}(\mathbb{R}^d)$; we may of course assume without loss of generality that the point under consideration is $0\in B$. A basis of neighborhoods of $0$ in $\mathcal{B}^{M_p,k}_\mathcal{W}(\R^d)$ is given by
$$
U(N,\varepsilon) =  \{ \varphi \in \mathcal{B}^{M_p,k}_\mathcal{W}(\R^d) \, : \, \|\varphi\|_{\mathcal{B}^{M_p,k}_{w_N}} \leq \varepsilon \}, \qquad  N \in \N, \varepsilon > 0.
$$
Let $N \in \N$ and $\varepsilon > 0$ be arbitrary. Let $C > 0$ be such that $\|\varphi\|_{\mathcal{B}^{M_p,h}_{w_N}} \leq C$ for all $\varphi \in B$. Next, choose $N_0 \in \N$ so large that $(k/h)^{N_0} \leq \varepsilon /C$. Finally, set $\delta = (l/k)^{N_0}\varepsilon$ and 
$$
V = \{ \varphi \in \mathcal{B}^{M_p,l}_\mathcal{W}(\R^d) \, : \, \|\varphi\|_{\mathcal{B}^{M_p,l}_{w_N}} \leq \delta \},
$$
a  neighborhood of $0$ in $\mathcal{B}^{M_p,l}_\mathcal{W}(\R^d)$. For $\varphi \in B \cap V$ we have that
\begin{align*}
\|\varphi\|_{\mathcal{B}^{M_p,k}_{w_N}} &= \max \left \{ \sup_{|\alpha| \leq N_0 } \frac{k^{|\alpha|} \| \varphi^{(\alpha)}\|_{w_N}}{M_\alpha}, \sup_{|\alpha| \geq N_0 } \frac{k^{|\alpha|} \| \varphi^{(\alpha)}\|_{w_N}}{M_\alpha}  \right \} \\
&\leq \max \{ (k/l)^{N_0} \delta, (k/h)^{N_0} C \} \leq \varepsilon,
\end{align*} 
which means that $B \cap V \subseteq U(N,\varepsilon)$.
\end{proof}
\begin{proposition} \label{boundedly-stable-1}
Let $M_p$ be a weight sequence and let  $\mathcal{V} = (v_n)_{n}$ be a decreasing weight system that is regularly decreasing. Then, $\mathcal{B}^{\ast}_\mathcal{V}(\R^d)$ is  boundedly stable.
\end{proposition}
\begin{proof}
We only prove the Beurling case, the Roumieu case is similar. Recall that, since $\mathcal{V}$ is regularly decreasing, $\mathcal{V}C(\R^d)$ is boundedly stable (cf.\ Subsection \ref{sect-reg-cont}). Let $n \in \N$ be arbitrary and choose $m \geq n$ such that for all $k \geq m$ the spaces $Cv_k(\R^d)$ and $Cv_m(\R^d)$ induce the same topology on the bounded sets of $Cv_n(\R^d)$. Now let $B$ be a bounded set of $\mathcal{B}^{(M_p)}_{v_n}(\R^d)$. We shall show that, for all $k \geq m$, the spaces  $\mathcal{B}^{(M_p)}_{v_m}(\R^d)$ and $\mathcal{B}^{(M_p)}_{v_k}(\R^d)$ induce the same topology on $B$. As before, we may assume that $0\in B$ and we only verify that the filter of neighborhoods of $0$ induced by $\mathcal{B}^{(M_p)}_{v_k}(\mathbb{R}^d)$ in $B$ is finer than that induced by $\mathcal{B}^{(M_p)}_{v_m}(\mathbb{R}^d)$.  A basis of neighborhoods of $0$ in $\mathcal{B}^{(M_p)}_{v_m}(\R^d)$ given by 
$$
U(h,\varepsilon) =  \{ \varphi \in \mathcal{B}^{(M_p)}_{v_m}(\R^d) \, : \, \|\varphi \|_{\mathcal{B}^{M_p,h}_{v_m}} \leq \varepsilon \}, \qquad h,\varepsilon > 0.
$$
Let $h,\varepsilon > 0$ be arbitrary. Since $B$ is a bounded subset of $\mathcal{B}^{(M_p)}_{v_n}(\R^d)$, the set 
$$
B' = \left \{ \frac{h^{|\alpha|}\varphi^{(\alpha)}}{M_\alpha} \, : \, \alpha \in \N^d, \varphi \in B \right \}
$$
is bounded in $Cv_n(\R^d)$. Consider the following neighborhood of $0$ in $Cv_m(\R^d)$
$$
U' = \{ f \in Cv_m(\R^d) \, : \, \|f\|_{v_m} \leq \varepsilon\}.
$$
As $Cv_m(\R^d)$ and $Cv_k(\R^d)$ induce the same topology on $B'$, there is $\delta > 0$ such that, for
$$
V' = \{ f \in Cv_k(\R^d) \, : \, \|f\|_{v_k} \leq \delta\},
$$
it holds that $B' \cap V' \subseteq U'$. Set 
$$
V = \{ \varphi \in \mathcal{B}^{(M_p)}_{v_k}(\R^d) \, : \, \|\varphi \|_{\mathcal{B}^{M_p,h}_{v_k}} \leq \delta \},
$$
a neighborhood of $0$ in $\mathcal{B}^{(M_p)}_{v_k}(\R^d)$. Let $\varphi \in B \cap V$. Then, $h^{|\alpha|}\varphi^{(\alpha)}/M_\alpha \in B' \cap V' \subseteq U'$ for all $\alpha \in \N^d$ and, thus, $\varphi \in U$. Whence $B \cap V \subseteq U(h,\varepsilon)$.
\end{proof}
For later use, we remark that Proposition \ref{boundedly-stable} and Proposition \ref{boundedly-stable-1} can be improved if we assume that $\mathcal{W} $ and $\mathcal{V}$ satisfy \eqref{decay-weights} and \eqref{decay-weights-1}, respectively. Namely, in such a case all of the spaces  $\mathcal{B}^*_\mathcal{W}(\R^d)$ and $\mathcal{B}^*_\mathcal{V}(\R^d)$ are even Schwartz. To prove this, we need the following simple lemma whose verification is left to the reader.
\begin{lemma}\label{compact-inclusion}
Let $M_p$ be a weight sequence, let $w$ and $v$ be positive continuous functions on $\R^d$ such that $v/w$ vanishes at $\infty$, and let $0 < k <h$. Then, the inclusion mapping $\mathcal{B}^{M_p,h}_w(\R^d) \rightarrow \mathcal{B}^{M_p,k}_v(\R^d)$ is compact.
\end{lemma}
Lemma \ref{compact-inclusion} yields:
\begin{proposition} \label{LFS-1}
Let $M_p$ be a weight sequence and let  $\mathcal{W} = (w_n)_{n}$ be an increasing weight system satisfying \eqref{decay-weights}. Then, $\mathcal{B}^{(M_p)}_\mathcal{W}(\R^d)$ is an $(FS)$-space while  $\mathcal{B}^{\{M_p\}}_\mathcal{W}(\R^d)$ is an $(LFS)$-space.
\end{proposition}
\begin{proposition} \label{LFS-2}
Let $M_p$ be a weight sequence and let  $\mathcal{V} = (v_n)_{n}$ be a decreasing weight system satisfying \eqref{decay-weights-1}. Then, the space $\mathcal{B}^{(M_p)}_\mathcal{V}(\R^d)$ is an $(LFS)$-space while $\mathcal{B}^{\{M_p\}}_\mathcal{V}(\R^d)$ is a $(DFS)$-space.
\end{proposition}
\begin{proof}
The fact that $\mathcal{B}^{\{M_p\}}_\mathcal{V}(\R^d)$ is a $(DFS)$-space follows directly from Lemma \ref{compact-inclusion}. Next, we consider the space $\mathcal{B}^{(M_p)}_\mathcal{V}(\R^d)$. We may assume without loss of generality that $v_{n+1}/v_n$ vanishes at $\infty$ for all $n \in \N$. Fix $n \in \N$ and set $v^{1}_n := \sqrt{v_nv_{n+1}}$. Notice that $v^{1}_n$ is a positive continuous function such that $v_{n+1} \leq v_n^{1} \leq v_n$ and $v_{n+1}/v^1_n$ vanishes at $\infty$. Hence we can inductively define a sequence $(v^N_n)_{N \in \N}$ of positive continuous functions such that $v_{n + 1} \leq v^1_{n} \leq \ldots \leq v^N_n \leq v^{N+1}_n \leq \ldots \leq v_{n}$ and $v^N_n/v^{N+1}_n$ vanishes at $\infty$ for each $N \in \N$. We have that
$$
\mathcal{B}^{(M_p)}_\mathcal{V}(\R^d) = \varinjlim_{n \in \N} \varprojlim_{N \in \N} \mathcal{B}^{M_p,N}_{v^N_n}(\R^d)
$$
as locally convex spaces. Moreover, the Fr\'echet spaces
$$
 \varprojlim_{N \in \N} \mathcal{B}^{M_p,N}_{v^N_n}(\R^d)
$$
are Schwartz because of Lemma \ref{compact-inclusion}.
\end{proof}
We are ready to study the regularity properties of the $(LF)$-spaces $\mathcal{B}^{\{M_p\}}_\mathcal{W}(\R^d)$ and  $\mathcal{B}^{(M_p)}_\mathcal{V}(\R^d)$. First we need to introduce some more notation.

Let $M_p$ be a non-quasianalytic weight sequence and let $K$ be a regular compact subset of $\R^d$. As customary \cite{Komatsu}, we denote by $\mathcal{D}^{M_p,h}_K$, $h >0$,  the Banach space consisting of all $\varphi \in C^\infty(\R^d)$ with $\operatorname{supp} \varphi \subseteq K$ such that
$$
\| \varphi \|_{\mathcal{D}^{M_p,h}_K} := \sup_{\alpha \in \N^d} \sup_{x \in K} \frac{h^{|\alpha|}|\varphi^{(\alpha)}(x)|}{M_\alpha} < \infty
$$
and set
$$
\mathcal{D}^{(M_p)}_K = \varprojlim_{h \to \infty}\mathcal{D}^{M_p,h}_K, \qquad  \mathcal{D}^{\{M_p\}}_K = \varinjlim_{h \to 0^+}\mathcal{D}^{M_p,h}_K.
$$
We shall also use the following mild condition on an increasing weight system $\mathcal{W} = (w_N)_N$:
\begin{equation}
\forall N \in \N \, \exists M \geq N \, \exists C > 0 \, \forall x \in \R^d \,: \, \sup_{y \in [-1,1]^d} w_N(x +y) \leq Cw_M(x).
\label{translation-weak}
\end{equation}

 \begin{theorem}\label{reg-Beurling}
 Let $M_p$ be a weight sequence and let $\mathcal{W} = (w_N)_{N}$ be an increasing weight system. Consider the following conditions:
\begin{itemize}
\item[$(i)$] $\mathcal{W}$ satisfies $(DN)$.
\item[$(ii)$] $\mathcal{B}^{\{M_p\}}_\mathcal{W}(\R^d)$ is boundedly retractive.
\item[$(iii)$] $\mathcal{B}^{\{M_p\}}_\mathcal{W}(\R^d)$ is complete.
\item[$(iv)$] $\mathcal{B}^{\{M_p\}}_\mathcal{W}(\R^d)$ is regular.
\item[$(v)$] $\mathcal{B}^{\{M_p\}}_\mathcal{W}(\R^d)$ satisfies $(wQ)$.
\end{itemize}
Then, $(i) \Rightarrow (ii)  \Rightarrow (iii)  \Rightarrow (iv)  \Rightarrow (v) \Rightarrow (ii)$. Moreover, if $M_p$ satisfies $(M.1)$, $(M.2)$, and $(M.3)$, and $\mathcal{W}$ satisfies \eqref{translation-weak}, then $(v) \Rightarrow (i)$. 
\end{theorem}
\begin{proof} The chain of implications
$(ii)  \Rightarrow (iii)  \Rightarrow (iv)  \Rightarrow (v)$ holds for general $(LF)$-spaces (cf.\ Subsection \ref{sect-reg-cond}), while $(v) \Rightarrow (ii)$ follows from Theorem \ref{reg-cond} and Proposition \ref{boundedly-stable}. We now show $(i) \Rightarrow (ii)$. Set $E = \mathcal{W}C(\R^d)$. We first assume that $E$ is normable. Then there is $N_0 \in \N$ such that $E$ is topologically isomorphic to $Cw_{N_0}(\R^d)$. Consequently, $\mathcal{B}^{\{M_p\}}_\mathcal{W}(\R^d)$ is topologically isomorphic to the $(LB)$-space $\mathcal{B}^{\{M_p\}}_{w_{N_0}}(\R^d)$. Applying Proposition \ref{boundedly-stable} to the constant weight system $\mathcal{W} = (w_{N_0})_{N \in \N}$, we obtain that  $\mathcal{B}^{\{M_p\}}_{w_{N_0}}(\R^d)$ is boundedly stable and, thus, boundedly retractive by Theorem \ref{reg-cond}. Next, we assume that $E$ is non-normable. By Theorem \ref{reg-cond} it suffices to show that $\mathcal{B}^{\{M_p\}}_\mathcal{W}(\R^d)$ is sequentially retractive. Let $(\varphi_j)_{j \in \N}$ be a null sequence in $\mathcal{B}^{\{M_p\}}_\mathcal{W}(\R^d)$. We set $X = \N^d$ (endowed with the discrete topology) and 
$$
\mathcal{V} = (v_n)_{n}, \qquad  v_n(\alpha) = n^{-|\alpha|}, \qquad \alpha \in \N^d.
$$
Notice that $\mathcal{V}$ satisfies $(\Omega)$. Clearly, the mapping 
$$
T: \mathcal{B}^{\{M_p\}}_\mathcal{W}(\R^d) \rightarrow \mathcal{V}C(\N^d;E): \varphi \mapsto \left( \frac{\varphi^{(\alpha)}}{M_\alpha}\right)_{\alpha \in \N^d}
$$
is continuous. Moreover, by Theorem \ref{completeness-F} and Proposition \ref{(DN)-Omega-1}, $\mathcal{V}C(\N^d;E)$ is sequentially retractive. Hence, there is $n \in \N$ such that the sequence $(T(\varphi_j))_{j}$ is contained and converges to zero in $Cv_n(\N^d;E)$, which precisely means that $(\varphi_j)_{j}$ is contained and converges to zero in $\mathcal{B}^{M_p,1/n}_{\mathcal{W}}(\R^d)$. Finally, we assume that $M_p$ satisfies $(M.1)$, $(M.2)$, and $(M.3)$, and that $\mathcal{W}$ satisfies \eqref{translation-weak}, and show $(v) \Rightarrow (i)$. By \cite[Cor.\ 4.10]{M-T} we have that $\mathcal{D}^{\{M_p\}}_{[-1,1]^d} \cong \Lambda'_0(\beta)$ as l.c.s., where $\beta = ({M(j^{1/d})})_{j}$. The sequence $\beta$ satisfies \eqref{stable} because of \cite[Lemma 4.1]{Komatsu}. Hence,  by Proposition \ref{W-fixed-1} and Remark \ref{S-2-inv-1}, it suffices to show that $(\mathcal{W}, \mathcal{D}^{\{M_p\}}_{[-1,1]^d})$ satisfies $(S_2)^*$. Since $\mathcal{B}^{\{M_p\}}_\mathcal{W}(\R^d)$ satisfies $(wQ)$, we have that
\begin{gather*}
\forall n \in \N \, \exists m > n \, \exists N \in \N \, \forall k > m \, \forall M \in \N \, \exists K \in \N \, \exists C > 0\, \forall \varphi \in \mathcal{B}^{M_p, 1/n}_\mathcal{W}(\R^d) \, : \\
\|\varphi\|_{\mathcal{B}^{M_p, 1/m}_{w_M}} \leq C\Big(\|\varphi\|_{\mathcal{B}^{M_p, 1/n}_{w_N}} + \|\varphi\|_{\mathcal{B}^{M_p, 1/k}_{w_K}}\Big).
\end{gather*}
Let $\varphi \in \mathcal{D}^{M_p,h}_{[-1,1]^d}$, $h > 0$, and $x \in \R^d$ be arbitrary. Then, the translation $T_x\varphi \in \mathcal{B}^{M_p,h}_\mathcal{W}(\R^{d})$ and it holds that
$$
 \inf_{y \in [-1,1]^d}w_N(x+y) \|\varphi\|_{\mathcal{D}^{M_p,h}_{[-1,1]^d}} \leq  \| T_x\varphi\|_{\mathcal{B}^{M_p,h}_{w_N}} \leq
 \sup_{y \in [-1,1]^d}w_N(x+y) \|\varphi\|_{\mathcal{D}^{M_p,h}_{[-1,1]^d}}
$$
for all $N \in \N$. Therefore, condition \eqref{translation-weak} implies that
\begin{gather*}
\forall n \in \N \, \exists m > n \, \exists N \in \N \, \forall k > m \, \forall M \in \N \, \exists K \in \N \, \exists C > 0\, \forall \varphi \in \mathcal{D}^{M_p, 1/n}_{[-1,1]^d} \, \forall x \in \R^d : \\
\|\varphi\|_{\mathcal{D}^{M_p,{1/m}}_{[-1,1]^d}} w_M(x) \leq C\Big(\|\varphi\|_{\mathcal{D}^{M_p,{1/n}}_{[-1,1]^d}}w_N(x) +  \|\varphi\|_{\mathcal{D}^{M_p,{1/k}}_{[-1,1]^d}} w_K(x)\Big).
\end{gather*}
\end{proof}
Let $\mathcal{V} = (v_n)_{n}$ be a decreasing weight system. We introduce the condition:
 \begin{equation}
\forall n \in \N \, \exists m \geq n \, \exists C > 0 \, \forall x \in \R^d \,: \, \sup_{y \in [-1,1]^d} v_m(x +y) \leq Cv_n(x).
 \label{translation-weak-1}
 \end{equation}

\begin{theorem}\label{reg-Roumieu-1}
Let $M_p$ be a weight sequence and let $\mathcal{V} = (v_n)_{n}$ be a decreasing weight system. Consider the following conditions:
\begin{itemize}
\item[$(i)$] $\mathcal{V}$ satisfies $(\Omega)$.
\item[$(ii)$] $\mathcal{B}^{(M_p)}_\mathcal{V}(\R^d)$ is boundedly retractive.
\item[$(iii)$] $\mathcal{B}^{(M_p)}_\mathcal{V}(\R^d)$ is complete.
\item[$(iv)$] $\mathcal{B}^{(M_p)}_\mathcal{V}(\R^d)$ is regular.
\item[$(v)$] $\mathcal{B}^{(M_p)}_\mathcal{V}(\R^d)$ satisfies $(wQ)$.
\end{itemize}
Then, $(i) \Rightarrow (ii)  \Rightarrow (iii)  \Rightarrow (iv) \Rightarrow (v)$ and, if $\mathcal{V}$ is regularly decreasing, $(v) \Rightarrow (ii)$. Moreover, if $M_p$ satisfies $(M.1)$, $(M.2)$, and $(M.3)$, and $\mathcal{V}$ satisfies \eqref{translation-weak-1}, then $(v) \Rightarrow (i)$. 
\end{theorem}
\begin{proof}
Again, the implications $(ii)  \Rightarrow (iii)  \Rightarrow (iv)  \Rightarrow (v)$ hold for general $(LF)$-spaces  (cf.\ Subsection \ref{sect-reg-cond}), while $(v) \Rightarrow (ii)$ (under the extra assumption that $\mathcal{V}$ is regularly decreasing) follows from Theorem \ref{reg-cond} and Proposition \ref{boundedly-stable-1}.  We now show $(i) \Rightarrow (ii)$.
By Theorem \ref{reg-cond} it suffices to show that $\mathcal{B}^{(M_p)}_\mathcal{V}(\R^d)$ is sequentially retractive.
Let $(\varphi_j)_{j \in \N}$ be a null sequence in $\mathcal{B}^{(M_p)}_\mathcal{V}(\R^d)$. We define $E$ as the Fr\'echet space consisting of all $(c_\alpha)_{\alpha} \in \C^{\N^d}$ such that
$$
 \sup_{\alpha \in \N^d} h^{|\alpha|} |c_\alpha| < \infty 
$$
for all $h > 0$. Notice that $E$ satisfies $(DN)$. Clearly, the mapping 
$$
T: \mathcal{B}^{(M_p)}_\mathcal{V}(\R^d) \rightarrow \mathcal{V}C(\R^d;E): \varphi \mapsto \left( \frac{\varphi^{(\alpha)}}{M_\alpha}\right)_{\alpha \in \N^d}
$$
is continuous. Moreover, by Theorem \ref{completeness-F} and Proposition \ref{(DN)-Omega-1}, $\mathcal{V}C(\N^d;E)$ is sequentially retractive. Hence there is $n \in \N$ such that the sequence $(T(\varphi_j))_{j}$ is contained and converges to zero in $Cv_n(\N^d;E)$, which precisely means that $(\varphi_j)_{j}$ is contained and converges to zero in $\mathcal{B}^{(M_p)}_{v_n}(\R^d)$. Next, we assume that  $M_p$ satisfies $(M.1)$, $(M.2)$, and $(M.3)$, and that $\mathcal{V}$ satisfies \eqref{translation-weak-1}, and show $(v) \Rightarrow (i)$. By \cite[Cor.\ 4.3]{M-T} we have that $\mathcal{D}^{(M_p)}_{[-1,1]^d} \cong \Lambda_\infty(\beta)$ as l.c.s., where $\beta = (e^{M(j^{1/d})})_{j}$. The sequence $\beta$ satisfies \eqref{stable} because of \cite[Lemma 4.1]{Komatsu}. Hence, by Proposition \ref{E-fixed-1} and Remark \ref{S-2-inv}, it suffices to show that $(\mathcal{D}^{(M_p)}_{[-1,1]^d}, \mathcal{V})$ satisfies $(S_2)^*$; but this can be established as in the last part of the proof of Theorem \ref{reg-Beurling}.
 \end{proof}
We end this section by giving several examples. Let $A_p$ be a weight sequence, we define the following weight systems
\begin{equation}
\mathcal{W}_{(A_p)} := (e^{A(N\cdot)})_{N \in \N}, \qquad \mathcal{W}_{\{A_p\}} := (e^{-A(\cdot /N)})_{N \in \N},
\label{WS-1}
\end{equation}
\begin{equation}
\mathcal{V}_{(A_p)} := (e^{-A(n\cdot)})_{n \in \N}, \qquad \mathcal{V}_{\{A_p\}} := (e^{A(\cdot /n)})_{n \in \N}.
\label{WS-2}
\end{equation}
We use the ensuing notation for the associated Gelfand-Shilov type spaces 
\begin{equation}
\mathcal{S}^{(M_p)}_{(A_p)}(\R^d) = \mathcal{B}^{(M_p)}_{\mathcal{W}_{(A_p)}}(\R^d), \qquad \mathcal{S}^{\{M_p\}}_{\{A_p\}}(\R^d) = \mathcal{B}^{\{M_p\}}_{\mathcal{V}_{\{A_p\}}}(\R^d),
\label{GS-1}
\end{equation}
\begin{equation}
\mathcal{S}^{(M_p)}_{\{A_p\}}(\R^d) := \mathcal{B}^{(M_p)}_{\mathcal{V}_{\{A_p\}}}(\R^d), \qquad \mathcal{S}^{\{M_p\}}_{(A_p)}(\R^d) := \mathcal{B}^{\{M_p\}}_{\mathcal{W}_{(A_p)}}(\R^d),
\label{GS-2}
\end{equation}
\begin{equation}
\mathcal{O}_C^{(M_p),(A_p)}(\R^d) := \mathcal{B}^{(M_p)}_{\mathcal{V}_{(A_p)}}(\R^d), \qquad \mathcal{O}_C^{\{M_p\},\{A_p\}}(\R^d):= \mathcal{B}^{\{M_p\}}_{\mathcal{W}_{\{A_p\}}}(\R^d),
\label{OC-1}
\end{equation}
\begin{equation}
\mathcal{O}_C^{(M_p),\{A_p\}}(\R^d) := \mathcal{B}^{(M_p)}_{\mathcal{W}_{\{A_p\}}}(\R^d), \qquad \mathcal{O}_C^{\{M_p\},(A_p)}(\R^d) := \mathcal{B}^{\{M_p\}}_{\mathcal{V}_{(A_p)}}(\R^d).
\label{OC-2}
\end{equation}
The spaces \eqref{GS-1} are the classical Gelfand-Shilov spaces already introduced in Subsection \ref{gelfand-shilov-spaces}, while the spaces \eqref{GS-2} may be viewed as mixed type Gelfand-Shilov spaces. The spaces \eqref{OC-1} and  \eqref{OC-2} are the natural analogue of the space $\mathcal{O}_C(\R^d)$ with respect to the spaces \eqref{GS-1} and  \eqref{GS-2}, respectively; we also mention that related spaces\footnote{Indeed our $\mathcal{O}_C^{(M_p),(M_p)}(\R^d)$ coincides with $\mathcal{O}_C^{(M_p)}(\R^d)$ from \cite{D-P-P-V}; on the other hand, it should be noticed that our space $\mathcal{O}_C^{
\{M_p\},\{M_p\}}(\R^d)$ differs from the one denoted by $\mathcal{O}_C^{\{M_p\}}(\R^d)$ in \cite{D-P-P-V}.}
  have been studied in \cite[Sect.\ 3]{D-P-P-V}.

In order to be able to apply Theorem \ref{reg-Beurling} and Theorem \ref{reg-Roumieu-1} to the special cases under consideration, we first study the properties of the weight systems  \eqref{WS-1} and  \eqref{WS-2}. As customary, we write $a_p:= A_p/A_{p-1}$, $p \geq 1$, and 
$$
a(t) := \sum_{a_p \leq t} 1, \qquad t \geq 0,
$$
for the counting function of the sequence $a_p$. If $A_p$ satisfies $(M.1)$, then \cite[Eq.\ (3.11)]{Komatsu}
\begin{equation}
A(t) = \int_0^{t} \frac{a(\lambda)}{\lambda} \dl, \qquad t \geq 0.
\label{representation-ass}
\end{equation}
\begin{lemma}
Let $A_p$ be a weight sequence satisfying $(M.1)$. Then,
\begin{itemize}
\item[$(i)$] $\mathcal{W}_{(A_p)}$ and $\mathcal{W}_{\{A_p\}}$ satisfy \eqref{decay-weights} and \eqref{translation-weak}, while $\mathcal{V}_{(A_p)}$ and $\mathcal{V}_{\{A_p\}}$ satisfy \eqref{decay-weights-1} and \eqref{translation-weak-1}.
\item[$(ii)$] $\mathcal{W}_{(A_p)}$ satisfies $(DN)$.
\item[$(iii)$] $\mathcal{V}_{\{A_p\}}$ satisfies $(\Omega)$.
\item[$(iv)$] $\mathcal{W}_{\{A_p\}}$ satisfies $(DN)$ if and only if 
\begin{equation}
\forall h > 0 \, \exists k > 0 \, \exists C > 0 \, \forall t \geq 0 \, : \, A(t) + A(kt) \leq 2A(ht) + C.
\label{Lang-cond}
\end{equation}
\item[$(v)$] If $A_p$ satisfies $(M.2)$, then $\mathcal{V}_{(A_p)}$ satisfies $(\Omega)$ if and only if $A_p$ satisfies $(M.2)^*$.
\end{itemize}
\end{lemma}
\begin{proof}
$(i)$ Obvious.

$(ii)$ It suffices to notice that for all $h > 0$ we have that
$$
2A(ht) = \sup_{p \in \N} \log \frac{(ht)^{2p}A_0^2}{A^2_p} =  \sup_{p \in \N} \left( \log \frac{t^pA_0}{A_p} + \log \frac{(h^2t)^pA_0}{A_p}\right) \leq A(t) + A(h^2t).
$$

$(iii)$ It suffices to show that 
$$
\forall h > 0 \, \exists k < h \, \forall l < k \, \exists C > 0 \, \forall t \geq 0 \, : \, (C+1)A(kt) \leq CA(ht) + A(lt),
$$
which is equivalent to
$$
A(kt) - A(lt) \leq C(A(ht) - A(kt)).
$$
Set $k = h/e$. By \eqref{representation-ass} we have that
$$
A((ht)/e) - A(lt) \leq \log(h/(le)) a((ht)/e) \leq  \log(h/(le))(A(ht) - A((ht)/e).
$$

$(iv)$ Clear.

$(v)$ The direct implication follows from the fact that $A_p$ satisfies $(M.2)^*$ if and only if (cf.\ Subsection \ref{gelfand-shilov-spaces})
$$
A(2t) \leq H'A(t) + \log C'_0, \qquad t \geq 0,
$$
for some $H',C'_0 \geq  1$. Conversely, assume that $A_p$ satisfies $(M.2)^*$.  It suffices to show that 
$$
\forall h > 0 \, \exists k > h \, \forall l > k \, \exists C, C' > 0 \, \forall t \geq 0 \, : \, -(C+1)A(kt) \leq -CA(ht) - A(lt) + C',
$$
which is equivalent to
$$
A(lt) - A(kt) \leq C(A(kt) - A(ht)) + C'.
$$
Condition $(M.2)^*$ implies that there $C, m > 0$ such that 
$$
a(lt) \leq ma(ht) + C, \qquad t \geq 0.
$$
Set $k = he$.  Hence
\begin{align*}
A(lt) -A(het) &\leq \log (l/(he))a(lt) \\
&\leq \log (l/(he))ma(ht) + \log (l/(he)) C \\
&\leq  \log (l/(he))m(A(het) - A(ht)) + \log (l/(he))C.
\end{align*}
\end{proof}

We then have,

\begin{corollary}
\label{cor 1 complete special cases}
Let $M_p$ and $A_p$ be weight sequences satisfying $(M.1)$. Then,
\begin{itemize}
\item[$(i)$] $\mathcal{S}^{(M_p)}_{\{A_p\}}(\R^d)$ and $\mathcal{S}^{\{M_p\}}_{(A_p)}(\R^d)$ are complete.
\item[$(ii)$] Assume that $A_p$ satisfies $(M.2)$. $\mathcal{O}_C^{(M_p),(A_p)}(\R^d)$ is complete if $A_p$ satisfies $(M.2)^\ast$. If, in addition, $M_p$ satisfies $(M.1)$, $(M.2)$, and $(M.3)$, then $\mathcal{O}_C^{(M_p),(A_p)}(\R^d)$ is complete if and only if  $A_p$ satisfies $(M.2)^\ast$. 
\item[$(iii)$] $\mathcal{O}_C^{\{M_p\},\{A_p\}}(\R^d)$ is complete if $A_p$ satisfies \eqref{Lang-cond}. If $M_p$ additionally satisfies $(M.1)$, $(M.2)$, and $(M.3)$, then $\mathcal{O}_C^{\{M_p\},\{A_p\}}(\R^d)$ is complete if and only if $A_p$ satisfies \eqref{Lang-cond}.
\end{itemize}
\end{corollary}

\begin{remark}\label{rk incomplete}
Condition \eqref{Lang-cond} is  satisfied by the so-called $q$-Gevrey sequences $A_p = q^{p^2}$, $q > 1$. On the other hand, it is very important to point out that if $A_p$ satisfies $(M.1)$ and $(M.2)$, then \eqref{Lang-cond} cannot hold for $A_p$. For example, this is always the case for the Gevrey sequences $A_p = p^\sigma$, $\sigma > 0$. We can thus supplement Corollary \ref{cor 1 complete special cases} as follows,

\end{remark}
\begin{corollary}
\label{cor 2 complete special cases}
Suppose that $M_p$ satisfies $(M.1)$, $(M.2)$, and $(M.3)$, while $A_p$ satisfies $(M.1)$ and $(M.2)$. Then, the space $\mathcal{O}_C^{\{M_p\},\{A_p\}}(\R^d)$ is incomplete.
\end{corollary}

\section{The ultradistribution spaces $\mathcal{B}'^*_\mathcal{W}(\R^d)$ and $\mathcal{B}'^*_\mathcal{V}(\R^d)$}\label{sect-duals}
The aim of this section is to study the dual spaces $\mathcal{B}'^*_\mathcal{W}(\R^d)$ and $\mathcal{B}'^*_\mathcal{V}(\R^d)$. Our first goal is to characterize these spaces in terms of the growth of the convolution averages of their elements. In order to do so, we start by studying the STFT on these spaces. Motivated by our convolution characterization, we introduce three natural locally convex topologies on the spaces $\mathcal{B}'^*_\mathcal{W}(\R^d)$ and $\mathcal{B}'^*_\mathcal{V}(\R^d)$ and show that they are all identical. Finally, with the aid of the results from Section \ref{sect-reg-ultra}, we give necessary and sufficient conditions for these spaces to be ultrabornological. 

We start by  introducing the class of weight systems that will be considered throughout this section.
Let $A_p$ be a weight sequence. An increasing weight system $\mathcal{W} = (w_N)_{N}$ is said to be \emph{$(A_p)$-admissible} if 
$$
\forall N \in \N \, \exists \lambda > 0 \, \exists M \geq N  \, \exists \, C > 0 \, \forall x,t \in \R^d \, : \,  w_N(x+t) \leq C w_M(x) e^{A(\lambda t)},
$$
while it is said to be \emph{$\{A_p\}$-admissible} if
$$
\forall N \in \N \, \forall \lambda > 0 \, \exists M \geq N  \, \exists \, C > 0 \, \forall x,t \in \R^d \, : \, w_N(x+t) \leq C w_M(x) e^{A(\lambda t)}.
$$
Likewise, a decreasing weight system $\mathcal{V} = (v_n)_{n}$ is said to be \emph{$(A_p)$-admissible} if 
$$
\forall n \in \N \, \exists \lambda > 0 \, \exists m \geq n  \, \exists \, C > 0 \, \forall x,t \in \R^d \, : \, v_m(x+t) \leq C v_n(x) e^{A(\lambda t)},
$$
while it is said to be \emph{$\{A_p\}$-admissible} if
$$
\forall n \in \N \, \forall \lambda > 0 \, \exists m \geq n  \, \exists \, C > 0 \, \forall x,t \in \R^d \, : \, v_m(x+t) \leq C v_n(x) e^{A(\lambda t)}.
$$
Let $B_p$ be a weight sequence such that $A_p \subset B_p$, then $\mathcal{W}_{(B_p)}$ is $(A_p)$-admissible while $\mathcal{W}_{\{B_p\}}$ is $\{A_p\}$-admissible. Similarly, $\mathcal{V}_{(B_p)}$ is $(A_p)$-admissible while $\mathcal{V}_{\{B_p\}}$ is $\{A_p\}$-admissible. We also need the following strengthened versions of \eqref{decay-weights} and \eqref{decay-weights-1}:
\begin{equation}
\forall N \in \N \, \exists M > N \, : \, w_N(x)/w_M(x) = O(|x|^{-(d+1)}),
\label{decay-weights-L1}
\end{equation}
and
\begin{equation}
\forall n \in \N \, \exists m > n \, : \, v_m(x)/v_n(x) = O(|x|^{-(d+1)}).
\label{decay-weights-L1-1}
\end{equation}
If $B_p$ is a weight sequence satisfying $(M.1)$ and $(M.2)$, then $\mathcal{W}_{(B_p)}$ and $\mathcal{W}_{\{B_p\}}$ satisfy \eqref{decay-weights-L1} while $\mathcal{V}_{(B_p)}$ and $\mathcal{V}_{\{B_p\}}$ satisfy \eqref{decay-weights-L1-1}.

Unless otherwise stated, $M_p$ and $A_p$ will from now on \emph{always} stand for weight sequences satisfying \eqref{group-cond}. On the other hand, $\mathcal{W} = (w_N)_{N}$ and $\mathcal{V} = (v_n)_{n}$ will \emph{always} denote an increasing and decreasing weight system satisfying \eqref{decay-weights-L1} and \eqref{decay-weights-L1-1}, respectively, which are assumed to be $(A_p)$-admissible in the Beurling case and $\{A_p\}$-admissible in the Roumieu case. 

\begin{lemma}\label{density-lemma}
Let $w$ and $v$ be positive continuous functions on $\R^d$ such that $v/w$ vanishes at $\infty$ and
\begin{equation}
\label{bound v-w for STFT}
v(x+t) \leq Cw(x)e^{A(\lambda t)}, \qquad x,t \in \R^d,
\end{equation}
for some $C, \lambda > 0$. Then, for $0 < kH < h$, the space $\mathcal{S}^{(M_p)}_{(A_p)}(\R^d)$ is dense in $\mathcal{B}^{M_p,h}_{w}(\R^d)$ with respect to the norm $\| \, \cdot \,  \|_{\mathcal{B}^{M_p,k}_v}$.
\end{lemma}
\begin{proof}
 Choose $\psi, \chi \in \mathcal{S}^{(M_p)}_{(A_p)}(\R^d)$ with $ \psi(0) = 1$ and $\int_{\R^d}\chi(x) \dx = 1$. We define $\psi_n =\psi(\cdot/n)$, $\chi_n = n^d\chi(n \cdot)$, and $\varphi_n = \chi_n \ast (\psi_n \varphi) \in \mathcal{S}^{(M_p)}_{(A_p)}(\R^d)$, $n \geq 1$. We now show that $\varphi_n \rightarrow \varphi$ in $\mathcal{B}^{M_p,k}_v(\R^d)$. Choose $l > 0$ so large that $h^{-1}+ l^{-1} \leq (kH)^{-1}$. Notice that
\begin{equation}
\| \varphi_n - \varphi\|_{\mathcal{B}^{M_p,k}_v} \leq  \| \varphi_n - \psi_n\varphi\|_{\mathcal{B}^{M_p,k}_v} + \| \psi_n\varphi - \varphi\|_{\mathcal{B}^{M_p,k}_v}. 
\label{triangle}
\end{equation} 
We start by estimating the second term in the right-hand side of \eqref{triangle}.  We have that
\begin{align*}
\| \psi_n\varphi - \varphi\|_{\mathcal{B}^{M_p,k}_v}  & \leq  \sup_{\alpha \in \N^d} \sup_{x \in \R^d} \frac{k^{|\alpha|}v(x)}{M_\alpha}|\psi(x/n)-1| |\varphi^{(\alpha)}(x)| \\
&+ \frac{1}{n}\sup_{\alpha \in \N^d} \sup_{x \in \R^d} \frac{k^{|\alpha|}v(x)}{M_\alpha}\sum_{\beta \leq \alpha, \beta \neq 0}\binom{\alpha}{\beta} |\psi^{(\beta)}(x/n)||
\varphi^{(\alpha-\beta)}(x)| \\
&\leq \|\varphi\|_{\mathcal{B}^{M_p,k}_w}\sup_{x \in \R^d}\frac{v(x)}{w(x)}|\psi(x/n)-1| + \frac{1}{n}\|\psi\|_{\mathcal{S}^{M_p,l}_{A_p,0}}\|\varphi\|_{\mathcal{B}^{M_p,h}_v},
\end{align*}
which tends to zero because $\psi(0) = 1$ and $v/w$ vanishes at $\infty$. Next, we estimate the first term at the right-hand side of \eqref{triangle}. Clearly,
\begin{align*}
\| \psi_n\varphi \|_{\mathcal{B}^{M_p,kH}_w}  &\leq \sup_{\alpha \in \N^d} \sup_{x \in \R^d} \frac{(kH)^{|\alpha|}w(x)}{M_\alpha}\sum_{\beta \leq \alpha}\binom{\alpha}{\beta} |\psi^{(\beta)}(x/n)||\varphi^{(\alpha-\beta)}(x)| \\
&\leq \|\psi\|_{\mathcal{S}^{M_p,l}_{A_p,0}}\|\varphi\|_{\mathcal{B}^{M_p,h}_w}
\end{align*}
for all $n\in \N$. Hence
\begin{align*}
&
\| \varphi_n - \psi_n\varphi\|_{\mathcal{B}^{M_p,k}_v}  \\
&
\leq \sup_{\alpha \in \N^d} \sup_{x \in \R^d}\frac{k^{|\alpha|}v(x)}{M_\alpha}\int_{\R^d} |\chi(t)| |(\psi_n\varphi)^{(\alpha)}(x-(t/n)) - (\psi_n\varphi)^{(\alpha)}(x)| \dt \\
&
\leq \frac{1}{n}\sup_{\alpha \in \N^d} \sup_{x \in \R^d}\frac{k^{|\alpha|}v(x)}{M_\alpha}\int_{\R^d} |\chi(t)||t| \sum_{j=1}^d \int_{0}^1|(\psi_n\varphi)^{(\alpha+e_j)}(x-(\gamma t/n))|\dgamma \dt \\
&\leq \frac{1}{n}dC\| \psi_n\varphi \|_{\mathcal{B}^{M_p,kH}_w}\sup_{\alpha \in \N^d} \sup_{x \in \R^d}\frac{k^{|\alpha|}M_{\alpha+1}}{(kH)^{|\alpha|+1}M_\alpha}\int_{\R^d} |\chi(t)||t| e^{A(\lambda t)}\dt \\
&\leq \frac{C'}{n}.
\end{align*}
\end{proof}

\begin{corollary}\label{dense-inclusion}
We have the following dense continuous inclusions
$$
\mathcal{S}^{\ast}_\dagger(\R^d) \hookrightarrow \mathcal{B}^{\ast}_{\mathcal{W}}(\R^d) \rightarrow \mathcal{W}C(\R^d) \hookrightarrow \mathcal{S}'^{\ast}_\dagger(\R^d)
$$
and 
$$
\mathcal{S}^{\ast}_\dagger(\R^d) \hookrightarrow \mathcal{B}^{\ast}_{\mathcal{V}}(\R^d) \rightarrow \mathcal{V}C(\R^d) \hookrightarrow \mathcal{S}'^{\ast}_\dagger(\R^d).
$$
\end{corollary}
 Corollary \ref{dense-inclusion} of course tells that we may view the dual spaces $\mathcal{B}'^{\ast}_{\mathcal{W}}(\R^d)$ and $\mathcal{B}'^{\ast}_{\mathcal{V}}(\R^d)$ as vector subspaces of  $\mathcal{S}'^{\ast}_\dagger(\R^d)$.  
\subsection{Characterization via the STFT}\label{Char STFT dual inductive} The goal of this subsection is to characterize $\mathcal{B}^{\ast}_{\mathcal{W}}(\R^d)$, $\mathcal{B}^{\ast}_{\mathcal{V}}(\R^d)$, and their dual spaces in terms of the  STFT.  We start with three lemmas.
\begin{lemma}\label{time-freq-norm}
Let $\psi \in \mathcal{S}^{(M_p)}_{(A_p)}(\R^d)$ and let $w$ and $v$ be positive functions on $\R^d$ for which \eqref{bound v-w for STFT} holds. Then,
$$
\|M_\xi T_x \psi \|_{\mathcal{B}^{M_p,h}_v} \leq C\|\psi\|_{\mathcal{S}^{M_p,2h}_{A_p,\lambda}} w(x)e^{M(4\pi h \xi)}.
$$
\end{lemma}
\begin{proof} We have that
\begin{align*}
&\|M_\xi T_x \psi \|_{\mathcal{B}^{M_p,h}_v} \\
&\leq \sup_{\alpha \in \N^d}\sup_{t \in \R^d} \frac{h^{|\alpha|}v(t)}{M_\alpha}\sum_{\beta \leq \alpha}\binom{\alpha}{\beta}(2\pi |\xi|)^{|\beta|}|\psi^{(\alpha - \beta)}(t-x)| \\
&\leq Cw(x)\sup_{\alpha \in \N^d}\sup_{t \in \R^d} \frac{1}{2^{|\alpha|}}\sum_{\beta \leq \alpha}\binom{\alpha}{\beta}\frac{(4\pi h |\xi|)^{|\beta|}}{M_\beta}\frac{(2h)^{|\alpha|-|\beta|}|\psi^{(\alpha - \beta)}(t-x)|e^{A(\lambda(t-x))}}{M_{\alpha-\beta}} \\
&\leq C\|\psi\|_{\mathcal{S}^{M_p,2h}_{A_p,\lambda}} w(x)e^{M(4\pi h \xi)}.
\end{align*}
\end{proof}

\begin{lemma}\label{STFT-test}
Let $\psi \in \mathcal{S}^{(M_p)}_{(A_p)}(\R^d)$ and let $w$ and $v$ be positive measurable functions on $\R^d$ for which \eqref{bound v-w for STFT} holds. Then, there is $C' > 0$ such that
$$
|V_\psi \varphi(x,\xi)| v(x)e^{M(\pi h\xi/\sqrt{d})}\leq C' \| \varphi\|_{\mathcal{B}^{M_p,h}_w}, \qquad \varphi \in \mathcal{B}^{M_p,h}_w(\R^d).
$$
\end{lemma}
\begin{proof}
For all $\alpha \in \N^d$ it holds that
\begin{align*}
|\xi^{\alpha}V_{\psi}\varphi(x, \xi)|v(x) &\leq \frac{C}{(2\pi)^{|\alpha|}} \sum_{\beta \leq \alpha} \binom{\alpha}{\beta} \int_{\R^d} |\varphi^{(\beta)}(t)|w(t) |\psi^{(\alpha - \beta)}(t-x)|e^{A(\lambda(t-x))} \dt \\
&\leq C' \|\varphi\|_{\mathcal{B}^{M_p,h}_w} (\pi h)^{-|\alpha|} M_\alpha.
\end{align*}
Hence, 
\begin{align*}
|V_{\psi}\varphi(x,\xi)|v(x) &\leq M_0C' \|\varphi\|_{\mathcal{B}^{M_p,h}_w} \inf_{p \in \N} \frac{M_p}{(\pi h|\xi|/\sqrt{d})^{p}M_0}  \leq M_0C'  \|\varphi\|_{\mathcal{B}^{M_p,h}_w} e^{-M(\pi h\xi/\sqrt{d})}.
\end{align*}
\end{proof}
\begin{lemma}\label{double-int-test}
Let $v$,$w$, and $u$ be positive measurable functions on $\R^d$ such that the pair $v,w$ satisfies \eqref{bound v-w for STFT}  and $w(x)/u(x) = O(|x|^{-d+1})$. Let $\psi \in \mathcal{S}^{(M_p)}_{(A_p)}(\R^d)$ and suppose that $F$ is a measurable function on $\R^{2d}$ such that 
$$
\sup_{(x,\xi) \in \R^{2d}}|F(x,\xi)|u(x)e^{M(k\xi)} < \infty.
$$ 
Then, the function
$$
t \rightarrow \int \int_{\R^{2d}} F(x,\xi) M_\xi T_x\psi(t) \dx \dxi
$$
belongs to $\mathcal{B}^{M_p,k/(4H\pi)}_v(\R^d)$.
\end{lemma}
\begin{proof}
Set $h = k/(4H\pi)$. Lemma \ref{time-freq-norm} implies that
\begin{align*}
&\left \| \int \int_{\R^{2d}} F(x,\xi) M_\xi T_x\psi \dx \dxi \right \|_{\mathcal{B}^{M_p,h}_v} \\
&\leq \int \int_{\R^{2d}} |F(x,\xi)| \| M_\xi T_x\psi \|_{\mathcal{B}^{M_p,h}_v} \dx \dxi \\
&\leq C\|\psi\|_{\mathcal{S}^{M_p,2h}_{A_p,\lambda}} \int \int_{\R^{2d}} |F(x,\xi)| w(x)e^{M(4\pi h \xi)}  \dx \dxi \\
&\leq CC' \|\psi\|_{\mathcal{S}^{M_p,2h}_{A_p,\lambda}} \int \int_{\R^{2d}} \frac{w(x)}{u(x)}e^{-M(k\xi) + M(k\xi/H)}  \dx \dxi \\
&\leq C_0CC' \|\psi\|_{\mathcal{S}^{M_p,2h}_{A_p,\lambda}} \int_{\R^{d}} \frac{w(x)}{u(x)}\dx \int_{\R^d}e^{-M(k\xi/H)}\dxi < \infty.
\end{align*}
\end{proof}
We are now able to characterize the spaces $\mathcal{B}^{\ast}_{\mathcal{W}}(\R^d)$, $\mathcal{B}^{\ast}_{\mathcal{V}}(\R^d)$, and their duals via the STFT.
\begin{proposition}\label{STFT-test-char}
Let $\psi \in \mathcal{S}^{(M_p)}_{(A_p)}(\R^d) \backslash \{0\}$ and $f \in \mathcal{S}'^\ast_\dagger(\R^d)$. Then, $f \in \mathcal{B}^{(M_p)}_{\mathcal{W}}(\R^d)$ ($f \in \mathcal{B}^{\{M_p\}}_{\mathcal{W}}(\R^d)$) if and only if 
\begin{equation}
\forall h > 0 \, \forall N \in \N \, (\exists h > 0 \, \forall N \in \N) \, : \, \sup_{(x,\xi) \in \R^{2d}}|V_\psi f(x,\xi)|w_N(x)e^{M(h\xi)} < \infty.
\label{bounds-STFT-test-char}
\end{equation}
\end{proposition}
\begin{proof}
The direct implication follows immediately from Lemma \ref{STFT-test}. Conversely, choose $\gamma \in \mathcal{S}^{(M_p)}_{(A_p)}(\R^d)$ such that $(\gamma, \psi)_{L^2} = 1$. By \eqref{regularization-via-STFT} we have that, for all $\varphi \in \mathcal{S}^\ast_\dagger(\R^d)$, it holds that
\begin{align*}
\langle f, \varphi \rangle &= \int \int_{\R^{2d}}V_\psi f(x, \xi) V_{\overline{\gamma}}\varphi(x, - \xi) \dx \dxi \\
&=  \int \int_{\R^{2d}}V_\psi f(x, \xi) \left(\int_{\R^{d}} \varphi(t) M_\xi T_x \gamma(t) \dt \right) \dx \dxi \\
&= \int_{\R^{d}}  \left( \int \int_{\R^{2d}} V_\psi f(x,\xi) M_\xi T_x \gamma(t) \dx \dxi \right) \varphi(t) \dt,
\end{align*}
where the switching of the integrals in the last step is permitted because of \eqref{bounds-STFT-test-char}. Hence,
$$
f =   \int \int_{\R^{2d}} V_\psi f(x,\xi) M_\xi T_x \gamma \dx \dxi 
$$
and we conclude that $f \in \mathcal{B}^{\ast}_{\mathcal{W}}(\R^d)$ by applying Lemma \ref{double-int-test} to $F = V_\psi f$.
\end{proof}
\begin{proposition}\label{STFT-char-1}
Let $\psi \in \mathcal{S}^{(M_p)}_{(A_p)}(\R^d) \backslash \{0\}$ and $f \in \mathcal{S}'^\ast_\dagger(\R^d)$. Then, $f \in \mathcal{B}^{(M_p)}_{\mathcal{V}}(\R^d)$ ($f \in \mathcal{B}^{\{M_p\}}_{\mathcal{V}}(\R^d)$) if and only if 
$$
\exists n \in \N \, \forall h>0 \, (\exists n \in \N \, \exists h> 0) \, : \, \sup_{(x,\xi) \in \R^{2d}}|V_\psi f(x,\xi)|v_n(x)e^{M(h\xi)} < \infty.
$$
\end{proposition}
\begin{proof}
This can be shown in the same way as Proposition \ref{STFT-test-char}.
\end{proof}

\begin{proposition}\label{STFT-dual}
Let $\psi \in \mathcal{S}^{(M_p)}_{(A_p)}(\R^d) \backslash \{0\}$ and $f \in \mathcal{S}'^\ast_\dagger(\R^d)$. Then, $f \in \mathcal{B}'^{(M_p)}_{\mathcal{W}}(\R^d)$ ($f \in \mathcal{B}'^{\{M_p\}}_{\mathcal{W}}(\R^d)$) if and only if 
$$
\exists h > 0 \, \exists N \in \N \, (\forall h > 0 \, \exists N \in \N) \, : \, \sup_{(x,\xi) \in \R^{2d}}\frac{|V_\psi f(x,\xi)|}{w_N(x)e^{M(h\xi)}} < \infty.
$$
\end{proposition}
\begin{proof} The direct implication follows immediately from Lemma \ref{time-freq-norm}. Conversely, choose $\gamma \in \mathcal{S}^{(M_p)}_{(A_p)}(\R^d)$ such that $(\gamma, \psi)_{L^2} = 1$. Lemma \ref{STFT-test} implies that 
$$
\langle \tilde{f}, \varphi \rangle := \int \int_{\R^{2d}}V_\psi f(x, \xi) V_{\overline{\gamma}}\varphi(x, - \xi) \dx \dxi, \qquad \varphi \in \mathcal{B}^\ast_\mathcal{W}(\R^d),
$$
defines a continuous linear functional on $\mathcal{B}^\ast_\mathcal{W}(\R^d)$. Since $\tilde{f}_{| \mathcal{S}^\ast_\dagger(\R^d)} = f$ (cf.\ \eqref{regularization-via-STFT}), we obtain that $f \in \mathcal{B}'^{\ast}_{\mathcal{W}}(\R^d)$.
\end{proof}

In a similar fashion one shows,

\begin{proposition}\label{STFT-dual-1}
Let $\psi \in \mathcal{S}^{(M_p)}_{(A_p)}(\R^d) \backslash \{0\}$ and $f \in \mathcal{S}'^\ast_\dagger(\R^d)$. Then, $f \in \mathcal{B}'^{(M_p)}_{\mathcal{V}}(\R^d)$ ($f \in \mathcal{B}'^{\{M_p\}}_{\mathcal{V}}(\R^d)$) if and only if 
$$
\forall n \in \N \, \exists h>0 \, (\forall n \in \N \, \forall h> 0) \, : \, \sup_{(x,\xi) \in \R^{2d}}\frac{|V_\psi f(x,\xi)|}{v_n(x)e^{M(h\xi)}} < \infty.
$$
\end{proposition}
\begin{corollary}\label{STFT-reg-dual-GS}
Let $\psi \in \mathcal{S}^{(M_p)}_{(A_p)}(\R^d) \backslash \{0\}$ and let $\gamma \in \mathcal{S}^{(M_p)}_{(A_p)}(\R^d)$ be a synthesis window for $\psi$, then
$$
\langle f, \varphi \rangle = \frac{1}{(\gamma, \psi)_{L^2}} \int \int_{\R^{2d}}V_\psi f(x, \xi) V_{\overline{\gamma}}\varphi(x, - \xi) \dx \dxi
$$
for all $f \in \mathcal{B}'^{\ast}_{\mathcal{W}}(\R^d)$ and $\varphi \in  \mathcal{B}^{\ast}_{\mathcal{W}}(\R^d)$ ($f \in \mathcal{B}'^{\ast}_{\mathcal{V}}(\R^d)$ and $\varphi \in  
\mathcal{B}^{\ast}_{\mathcal{V}}(\R^d)$), where the integral at the right-hand side is absolutely convergent.
\end{corollary}
As an application of Corollary \ref{STFT-reg-dual-GS} we now give a projective description of the space $\mathcal{B}^{\{M_p\}}_{\mathcal{V}}(\R^d)$ (cf.\ \cite{Pilipovic-94}),  a result that will be 
used later on. We need some preparation.
Let $\mathcal{V} = (v_n)_{n}$ be a decreasing weight system on a completely regular Hausdorff space $X$. The \emph{maximal Nachbin family associated with $\mathcal{V}$}, denoted by $\overline{V}(\mathcal{V}) = \overline{V}$, is given by the space of all nonnegative upper semicontinuous functions $v$ on $X$ such that 
$$
\sup_{x \in X} \frac{v(x)}{v_n(x)} < \infty
$$ 
for all $n \in \N$. The \emph{projective hull of $\mathcal{V}C(X)$} is then defined as 
$$
C\overline{V}(X) = \varprojlim_{v \in \overline{V}} Cv(X),
$$
where $Cv(X)$ is the seminormed space consisting of all $f \in C(X)$ such that 
$$
\|f\|_v := \sup_{x \in X} |f(x)|v(x) < \infty.
$$
The spaces $\mathcal{V}C(X)$ and $C\overline{V}(X)$ are equal as sets and have the same bounded sets \cite{Bierstedt} (see also Lemma \ref{proj-desc-1} below). The  problem of projective description in this context is to find conditions on $\mathcal{V}$ which ensure that these spaces coincide topologically. If $\mathcal{V}$ is regularly decreasing, this is always the case\ \cite[Cor.\ 9, p.\ 121]{Bierstedt}. We now collect some useful facts about the maximal Nachbin family associated with $\mathcal{V}$.

\begin{lemma}\label{proj-desc-1}
Let $f$ be a nonnegative function on $X$ and let $\mathcal{V} = (v_n)_{n}$ be a decreasing weight system. Then,
\begin{itemize}
\item[$(i)$] $\sup_{x \in X}f(x)v_n(x) < \infty$ for some $n \in \N$ if and only if $ \sup_{x \in X}f(x)v(x) < \infty$ for all $v \in \overline{V}$.
\item[$(ii)$] $\sup_{x \in X}f(x)/v(x) < \infty$ for some $v \in \overline{V}$ if and only if $\sup_{x \in X}f(x)/v_n(x) < \infty$ for all $n \in \N$. 
\end{itemize}
\end{lemma}
\begin{proof}
The direct implications are clear, we only need to show the ``if" parts. 

$(i)$ Suppose that $\sup_{x \in X}f(x)v_n(x) < \infty$ does not hold for any $n \in \N$. Choose a sequence $(x_k)_{k \in \N}$ of points in $X$ such that $f(x_k)v_k(x_k) \geq k$ for all $k \in \N$. Set
$C_n = \sup_{k \leq n} v_k(x_k)/v_n(x_k)$, $n \in \N$, and $v = \inf_{n \in \N} C_n v_n \in \overline{V}$. Since
$$
f(x_k)C_n v_n(x_k) \geq k, \qquad k,n \in \N,
$$
we have that $\sup_{x \in X}f(x)v(x) = \infty$, a contradiction.

$(ii)$ There are $C_n > 0$ such that $f \leq C_n v_n$ for all $n \in \N$. Then, the function $v = \inf_{n \in \N} C_nv_n \in \overline{V}$ satisfies the requirement.
\end{proof}
Likewise,
\begin{lemma}\label{proj-desc-2}
Let $X$ and $Y$ be completely regular Hausdorff spaces, let $f$ be a nonnegative function on $X \times Y$, let $\mathcal{V} = (v_n)_{n}$ be a decreasing weight system on $X$, and let $\mathcal{U} = (u_n)_{n}$ be a decreasing weight system on $Y$. Then,
\begin{itemize}
\item[$(i)$] $\sup_{(x,y) \in X \times Y}f(x,y)v_n(x)u_n(y) < \infty$ for some $n \in \N$ if and only if \newline $ \sup_{(x,y) \in X \times Y}f(x,y)v(x)u(y) < \infty$ for all $v \in \overline{V}(\mathcal{V})$ and $u \in \overline{V}(\mathcal{U})$.
\item[$(ii)$]$ \sup_{(x,y) \in X \times Y}f(x,y)/(v(x)u(y)) < \infty$ for some $v \in \overline{V}(\mathcal{V})$ and $u \in \overline{V}(\mathcal{U})$ if and only if  
 $\sup_{(x,y) \in X \times Y}f(x,y)/(v_n(x)u_n(y)) < \infty$ for all $n \in \N$. 
 \end{itemize}
\end{lemma}

\begin{lemma}\label{proj-desc-3} Let $w$ be a positive function on $X$ and let  $\mathcal{V} = (v_n)_{n}$ be a decreasing weight system satisfying
$$
\forall n \in \N \, \exists m \geq n \, \exists C > 0 \, \forall x \in X \,  : \, v_m(x) \leq Cw(x)v_n(x).
$$
Then,
$$
\forall v \in \overline{V} \, \exists  \overline{v} \in \overline{V}  \, \forall x \in X \,  : \,  v(x) \leq w(x)\overline{v}(x).
$$
\end{lemma}
\begin{proof}
Let $(n_k)_{k \in \N}$ be an increasing sequence of natural numbers such that $v_{n_{k+1}} \leq C_kwv_{n_k}$ for all $k \in \N$ and some $C_k > 0$. Next, choose $C'_k > 0$ such that $v \leq C'_k v_{n_k}$ for all $k \in \N$. Set $\overline{v} = \inf_{k \in \N} C_kC'_{k+1}v_{n_k} \in \overline{V}$. We have that
$$
v \leq \inf_{k \in \N} C'_{k+1}v_{n_{k+1}} \leq w  \inf_{k \in \N} C_kC'_{k+1}v_{n_{k}} = w\overline{v}.
$$
\end{proof}
Similarly, 
\begin{lemma}\label{proj-desc-4} Let $A_p$ be a weight sequence and let $\mathcal{V} = (v_n)_{n}$ be an  $\{A_p\}$-admissible decreasing weight system. Then,
$$
\forall v \in \overline{V} \, \forall \lambda > 0 \, \exists  \overline{v} \in \overline{V} \, \forall x,t \in \R^d \,  : \,  v(x+t) \leq \overline{v}(x)e^{A(\lambda t)}.
$$
\end{lemma}

We write $\mathcal{R}$ for the set of all increasing sequences $(h_j)_{j \in \N}$ of positive numbers tending to infinity. There is a natural order on $\mathcal{R}$ defined by $h_j \prec k_j$ if $h_j\leq k_j$ for all $j \in \N$, and with it $(\mathcal{R},\prec)$ becomes a directed set. Given a weight sequence $M_p$ 
with associated function $M$ and $h_j \in \mathcal{R}$, we denote by $M^{h_j}$ and $M_{h_j}$ the associated function of the sequence $M_p / \prod_{j = 0}^ph_j$  and $M_p \prod_{j = 0}^p h_j$, respectively. 
\begin{lemma}\label{projectivedescriptionlemma} Let $M_p$ be a weight sequence satisfying $(M.1)$ and $(M.2)$. Then,
\begin{itemize}
\item[$(i)$] for every $v \in \overline{V}(\mathcal{V}_{(M_p)})$ there is $h_j \in \mathcal{R}$ such that  
$$
\sup_{\xi \in \R^d} v(\xi)e^{M^{h_j}(\xi)} < \infty.
$$
\item[$(ii)$] for every $v \in \overline{V}(\mathcal{V}_{\{M_p\}})$ there is $h_j \in \mathcal{R}$ such that  
$$
\sup_{\xi \in \R^d} v(\xi)e^{-M_{h_j}(\xi)} < \infty.
$$
\end{itemize}
\end{lemma}
\begin{proof}
$(i)$ We first show that there is a \emph{supordinate} function $\varepsilon$, i.e., a continuous strictly increasing function $\varepsilon: [0,\infty) \rightarrow [0,\infty)$ with $\varepsilon(0) = 0$ and $ t= o(\varepsilon (t))$, such that 
\begin{equation}
\sup_{\xi \in \R^d} v(\xi)e^{M(\varepsilon(|\xi|))} < \infty.
\label{condition-supordinate}
\end{equation}
 Conditions $(M.1)$ and $(M.2)$ imply that for all $h> 0$ there is $t> 0$ such that $v(\xi) \leq e^{-M(h\xi)}$ for all $|\xi| > t$. Hence we can inductively determine a strictly increasing sequence $(t_n)_{n \in \N}$ with $t_0 = 0$ and $t_n \rightarrow \infty$ that satisfies 
$$ 
v(\xi) \leq e^{-M((n+1)\xi)}, \qquad |\xi| \geq t_n, n \geq 1.
$$
Let $l_n$ denote the line through the points $(t_n, nt_n)$ and $(t_{n+1}, (n+1)t_{n+1})$, and define
$\varepsilon(t) = l_n(t)$ for $t \in [t_n, t_{n+1})$. The function $\varepsilon$ is supordinate and satisfies \eqref{condition-supordinate}.

Therefore it suffices to show that for every supordinate function $\varepsilon$ there is a sequence $h_j \in \mathcal{R}$ and $C > 0$ such that $M^{h_j}(t) \leq M(\varepsilon(t)) + C$ for all $t \geq 0$. We may assume without loss of generality that $\varepsilon(t)/t$ tends increasingly  to $\infty$, for otherwise we can find another supordinate function $\tilde \varepsilon$ with $\tilde \varepsilon \leq \varepsilon$ that does satisfy this property.   Define $h_p$, $p \geq 1$, as the unique solution of $\varepsilon(m_p/h_p) = m_p$ and $h_0 = h_1$. The sequences $m_p/h_p$ and $h_p$ tend increasingly to infinity. Let $t \geq \max \{m_1/h_1, M_1h_1/M^2_0\}$ be arbitrary and suppose that $m_p/h_p \leq t \leq m_{p+1}/h_{p+1}$ for some $p \geq 1$. Since $h_p t \leq \varepsilon(t)$, we have that
\begin{align*}
M^{h_j}(t) &\leq \sup_{q \geq 1} \log \prod_{j = 1}^q \frac{h_jt}{m_j} + \log h_0 = \log \prod_{j = 1}^p \frac{h_jt}{m_j} + \log h_0 \leq \log \prod_{j = 1}^p \frac{\varepsilon(t)}{m_j}  + \log h_0 \\
&\leq M(\varepsilon(t)) + \log h_0 .
\end{align*}
$(ii)$ See \cite[Lemma 4.5$(i)$]{D-V-VEmbeddingUltra16}. 
\end{proof}
We are now able to state and prove the projective description of the space $\mathcal{B}^{\{M_p\}}_{\mathcal{V}}(\R^d)$.
\begin{proposition}\label{projectivedescription} A function $\varphi \in C^\infty(\R^d)$ belongs to $\mathcal{B}^{\{M_p\}}_{\mathcal{V}}(\R^d)$ if and only if 
$$
\| \varphi \|_{\mathcal{B}^{M_p,h_j}_{v}} := \sup_{\alpha \in \N^d}\sup_{x \in \R^d} \frac{|\varphi^{(\alpha)}(x)|v(x)}{M_{\alpha}\prod_{j = 0}^{|\alpha|}h_j} < \infty
$$
for all $v \in \overline{V}$ and $h_j  \in \mathcal{R}$. Moreover, the topology of $\mathcal{B}^{\{M_p\}}_{\mathcal{V}}(\R^d)$ is generated by the 
system of seminorms $\{\| \, \cdot \,  \|_{\mathcal{B}^{M_p,h_j}_{v}} \, : \, v \in \overline{V}, h_j  \in \mathcal{R} \}$.
\end{proposition}
\begin{proof}
The first part follows from Lemma \ref{proj-desc-2} and the fact that for a sequence of positive numbers $(a_j)_{j \in \N}$ it holds that $\sup_{j \in \N} a_j n^j < \infty$ for all $n \in \N$ 
if and only if $\sup_{j \in \N} a_j \prod_{m = 0}^jh_m < \infty$ for some $h_j \in \mathcal{R}$ \cite[Lemma 3.4]{Komatsu3}. Next, we show the topological assertion. Clearly, every seminorm $\| 
\, \cdot \,  \|_{\mathcal{B}^{M_p,r_j}_{v}}$ acts continuously on $\mathcal{B}^{\{M_p\}}_{\mathcal{V}}(\R^d)$. Conversely, let $p$ be an arbitrary seminorm on $\mathcal{B}^{\{M_p\}}_{\mathcal{V}}
(\R^d)$. There is a strongly bounded set $B \subset \mathcal{B}'^{\{M_p\}}_{\mathcal{V}}(\R^d)$ such that
$$
p(\varphi) \leq \sup_{f \in B} |\langle f, \varphi \rangle|, \qquad \varphi \in \mathcal{B}^{\{M_p\}}_{\mathcal{V}}(\R^d).
$$
Let $\psi, \gamma \in \mathcal{S}^{(M_p)}_{(A_p)}(\R^d) \backslash \{0\}$ with $(\gamma, \psi)_{L^2} = 1$.  Lemma \ref{time-freq-norm} implies that 
$$
\sup_{f \in B} \sup_{(x,\xi) \in \R^{2d}}\frac{|V_\psi f(x,\xi)|}{v_n(x)e^{M(\xi/n)}} < \infty
$$
for all $n \in \N$. Lemma \ref{proj-desc-2} and Lemma \ref{projectivedescriptionlemma} therefore imply that
$$
|V_\psi f(x,\xi)| \leq v(x)e^{M_{h_j}(\xi)}, \qquad f \in B,
$$
for some $v \in \overline{V}$ and $h_j  \in \mathcal{R}$. By \cite[Lemma 2.3]{Prangoski} there is $h'_j \in \mathcal{R}$ such that $h'_j \leq h_j$ for $j$ large enough and 
$$
\prod_{j = 0}^{p+q}h'_j \leq 2^{p+q} \prod_{j = 0}^{p}h'_j \prod_{j = 0}^{q}h'_j, \qquad p,q \in \N,
$$ 
which implies that the weight sequence $M_p\prod_{j = 0}^ph'_j$ satisfies $(M.1)$ and $(M.2)$ (with the constant $2H$ instead of $H$). Furthermore, by Lemma \ref{proj-desc-3}, there is $\overline{v} \in \overline{V}$ such that $v(x)/\overline{v}(x) = O(|x|^{-(d+1}))$ while Lemma \ref{proj-desc-4} yields the existence of $\overline{\overline{v}} \in \overline{V}$ such that
$$
\overline{v}(x+t) \leq \overline{\overline{v}}(x)e^{A(t)}, \qquad x,t \in \R^d.
$$
Hence Corollary \ref{STFT-reg-dual-GS} and Lemma \ref{STFT-test} imply that
\begin{align*}
\sup_{f \in B} |\langle f, \varphi \rangle| &\leq \sup_{f \in B}  \int \int_{\R^{2d}}|V_\psi f(x, \xi)| |V_{\overline{\gamma}}\varphi(x, - \xi)| \dx \dxi \\
&\leq C\int \int_{\R^{2d}} v(x)e^{M_{h'_j}(\xi)} |V_{\overline{\gamma}}\varphi(x, - \xi)| \dx \dxi \\
&\leq C'\sup_{(x,\xi) \in \R^{2d}} |V_{\overline{\gamma}}\varphi(x,\xi)| \overline{v}(x)e^{M_{h'_j/(2H)}(\xi)} \\
&\leq C''\| \varphi \|_{\mathcal{B}^{M_p,\pi h'_j/(2H\sqrt{d})}_{\overline{\overline{v}}}}. 
\qedhere
\end{align*}
\end{proof}
\subsection{Characterization via convolution averages}\label{conv average subsection} In this subsection we characterize the elements $f$ of $\mathcal{B}'^{\ast}_{\mathcal{W}}(\R^d)$ and $\mathcal{B}'^{\ast}
_{\mathcal{V}}(\R^d)$ in terms of the growth of the convolution averages $f \ast \varphi$ with respect to $\varphi \in \mathcal{S}^\ast_\dagger(\R^d)$. We start with the following simple lemma whose verification is left to the reader. 
\begin{lemma}\label{translation-norm}
Let $w$ and $v$ be positive functions on $\R^d$ satisfying \eqref{bound v-w for STFT}. Then,
$$
\| T_x\varphi \|_{\mathcal{B}^{M_p,h}_v} \leq Cw(x) \|\varphi\|_{\mathcal{S}^{M_p,h}_{A_p,\lambda}}, \qquad   \varphi \in \mathcal{S}^{M_p,h}_{A_p,\lambda}(\R^d).
$$
\end{lemma}
\begin{lemma}\label{STFT-conv}
Let $\psi \in \mathcal{S}^{(M_p)}_{(A_p)}(\R^d)$. For every $h_j \in \mathcal{R}$ ($h > 0$) there is a bounded set $B \subset \mathcal{S}
^{(M_p)}_{(A_p)}(\R^d)$ ($B \subset \mathcal{S}^{\{M_p\}}_{\{A_p\}}(\R^d)$) such that
$$ 
|V_\psi f(x,\xi)|e^{-M^{h_j}(\xi)} \leq \sup_{\varphi \in B} |(f \ast \varphi)(x)|, \qquad f \in \mathcal{S}'^{(M_p)}_{(A_p)}(\R^d),
$$
$$
\left(|V_\psi f(x,\xi)|e^{-M(h\xi)} \leq \sup_{\varphi \in B} |(f \ast \varphi)(x)|, \qquad f \in \mathcal{S}'^{\{M_p\}}_{\{A_p\}}(\R^d) \right).
$$
\end{lemma}
\begin{proof} The set
$$
B = \{  e^{-M^{h_j}(\xi)}M_\xi \check{\overline{\psi}} \, :  \, \xi \in \R^d \}  \qquad \left( B = \{  e^{-M(h\xi)} M_\xi \check{\overline{\psi}} \, :  \, \xi \in \R^d \} \right)
$$
is bounded in $\mathcal{S}^{(M_p)}_{(A_p)}(\R^d)$ (in $\mathcal{S}^{\{M_p\}}_{\{A_p\}}(\R^d)$). Hence
$$ 
|V_\psi f(x,\xi)|e^{-M^{h_j}(\xi)} = |(f \ast M_\xi \check{\overline{\psi}})(x)|e^{-M^{h_j}(\xi)} \leq \sup_{\varphi \in B} |(f \ast \varphi)(x)|
$$
$$
\left(|V_\psi f(x,\xi)|e^{-M(h\xi)} = |(f \ast M_\xi \check{\overline{\psi}})(x)| e^{-M(h\xi)} \leq \sup_{\varphi \in B} |(f \ast \varphi)(x)|\right)
$$
for all $f \in \mathcal{S}'^{(M_p)}_{(A_p)}(\R^d)$ ($f \in \mathcal{S}'^{\{M_p\}}_{\{A_p\}}(\R^d)$).
\end{proof}

Given an increasing (decreasing) weight system $\mathcal{W} = (w_N)_{N}$ ($\mathcal{V} = (v_n)_{n}$) we define its \emph{dual weight system} as the decreasing (increasing) weight system given by $\mathcal{W}^\circ := (w^{-1}_N)_{N}$ ($\mathcal{V}^\circ := (v^{-1}_n)_{n}$). 

\begin{theorem}\label{conv-average-Beurling}
Let $f \in \mathcal{S}'^\ast_\dagger(\R^d)$. The following statements are equivalent:
\begin{itemize}
\item[$(i)$] $f \in \mathcal{B}'^{\ast}_{\mathcal{W}}(\R^d)$.
\item[$(ii)$] $f \ast \varphi \in \mathcal{W}^\circ C(\R^d)$ for all $\varphi \in \mathcal{S}^\ast_\dagger(\R^d)$.
\item[$(iii)$] $f \ast \varphi \in \mathcal{W}^\circ C(\R^d)$ for all $\varphi \in \mathcal{S}^\ast_\dagger(\R^d)$ and the mapping
$$
\ast_f : \mathcal{S}^\ast_\dagger(\R^d) \rightarrow \mathcal{W}^\circ C(\R^d): \varphi \rightarrow f \ast \varphi
$$
is continuous.
\item[$(iv)$] $f \ast \varphi \in \mathcal{B}^\ast_{\mathcal{W}^\circ}(\R^d)$ for all $\varphi \in \mathcal{S}^\ast_\dagger(\R^d)$.
\item[$(v)$] $f \ast \varphi \in \mathcal{B}^\ast_{\mathcal{W}^\circ}(\R^d)$ for all $\varphi \in \mathcal{S}^\ast_\dagger(\R^d)$ and the mapping
$$
\ast_f : \mathcal{S}^\ast_\dagger(\R^d) \rightarrow \mathcal{B}^\ast_{\mathcal{W}^\circ}(\R^d): \varphi \rightarrow f \ast \varphi
$$
is continuous.
\end{itemize}
\end{theorem}
\begin{proof}
$(i) \Rightarrow (ii)$ Consequence of Lemma \ref{translation-norm}.

$(ii) \Rightarrow (iii)$ Follows from De Wilde's closed graph theorem and the fact that the mapping
$\ast_f:  \mathcal{S}^\ast_\dagger(\R^d) \rightarrow  \mathcal{S}'^\ast_\dagger(\R^d)$ is continuous.

$(iii) \Rightarrow (iv)$ We first show the Beurling case. By Grothendieck's factorization theorem there is $N \in \N$ such that the mapping
$\ast_f: \mathcal{S}^{(M_p)}_{(A_p)}(\R^d) \rightarrow Cw^{-1}_N(\R^d)$ is well-defined and continuous. Let $\varphi \in \mathcal{S}^{(M_p)}_{(A_p)}(\R^d)$ be arbitrary. For every $h> 0$ the set
$$
B_h = \left\{ \frac{h^{|\alpha|}\varphi^{(\alpha)}}{M_\alpha} \, : \, \alpha \in \N^d \right\}
$$
is bounded in $\mathcal{S}^{(M_p)}_{(A_p)}(\R^d)$. Hence
$$
\| f \ast \varphi \|_{\mathcal{B}^{M_p,h}_{w_N^{-1}}} = \sup_{\alpha \in \N^d} \sup_{x \in \R^d}\frac{h^{|\alpha|}|(f\ast \varphi^{(\alpha)})(x)|}{M_\alpha w_N(x)} = \sup_{\chi \in B_h}\| f \ast \chi\|_{w^{-1}_N} < \infty.
$$
Next, we consider the Roumieu case. Let $\varphi \in \mathcal{S}^{\{M_p\}}_{\{A_p\}}(\R^d)$ be arbitrary and suppose that  $\varphi \in \mathcal{S}^{M_p,h}_{A_p,\lambda}(\R^d)$ for some $h, \lambda > 0$. The set
$$
B = \left\{ \frac{(h/H)^{|\alpha|}\varphi^{(\alpha)}}{M_\alpha} \, : \, \alpha \in \N^d \right\}
$$
is bounded in $\mathcal{S}^{M_p,h/H}_{A_p,\lambda}(\R^d)$. There is $N \in \N$  such that the mapping $\ast_f: \mathcal{S}^{M_p,h/H}_{A_p,\lambda}(\R^d) \rightarrow Cw^{-1}_N(\R^d)$ is well-defined and continuous. Hence
$$
\| f \ast \varphi \|_{\mathcal{B}^{M_p,h/H}_{w_N^{-1}}} = \sup_{\alpha \in \N^d} \sup_{x \in \R^d}\frac{(h/H)^{|\alpha|}|(f\ast \varphi^{(\alpha)})(x)|}{M_\alpha w_N(x)} = \sup_{\chi \in B}\| f \ast \chi \|_{w^{-1}_N} < \infty.
$$

$(iv) \Rightarrow (v)$ Follows from De Wilde's closed graph theorem and the fact that the mapping
$\ast_f:  \mathcal{S}^\ast_\dagger(\R^d) \rightarrow  \mathcal{S}'^\ast_\dagger(\R^d)$ is continuous.

$(v) \Rightarrow (iii)$ Obvious.

$(iii) \Rightarrow (i)$ Let $\psi \in \mathcal{S}^{(M_p)}_{(A_p)}(\R^d) \backslash \{0\}$. By Proposition \ref{STFT-dual} it suffices to show
$$
\exists h > 0 \, \exists N \in \N \, ( \forall h > 0 \, \exists N \in \N) \,: \, \sup_{(x,\xi) \in \R^{2d}}\frac{|V_\psi f(x,\xi)|}{w_N(x)e^{M(h\xi)}} < \infty.
$$
We first show the Beurling case. The above estimate is equivalent to (cf.\ Lemma \ref{proj-desc-1} and Lemma \ref{projectivedescriptionlemma})
$$
\exists N \in \N \, \forall h_j \in \mathcal{R} \, : \, \sup_{(x,\xi) \in \R^{2d}}\frac{|V_\psi f(x,\xi)|}{w_N(x)e^{M^{h_j}(\xi)}} < \infty.
$$
By Grothendieck's factorization theorem there is $N \in \N$ such that the mapping $\ast_f: \mathcal{S}^{(M_p)}_{(A_p)}(\R^d) \rightarrow Cw^{-1}_N(\R^d)$ is well-defined and continuous. Lemma \ref{STFT-conv} implies that for every $h_j \in \mathcal{R}$ there is a bounded set $B \subset \mathcal{S}^{(M_p)}_{(A_p)}(\R^d)$ such that
$$
|V_\psi f(x,\xi)|e^{-M^{h_j}(\xi)} \leq \sup_{\varphi \in B} |(f \ast \varphi)(x)| \leq \sup_{\varphi \in B}\| f \ast \varphi\|_{w^{-1}_N} w_N(x).
$$
Next, we consider the Roumieu case. Let $h > 0$ be arbitrary. By Lemma \ref{STFT-conv} we find a bounded set $B \subset \mathcal{S}^{\{M_p\}}_{\{A_p\}}(\R^d)$ such that
$$
|V_\psi f(x,\xi)|e^{M(h\xi)} \leq \sup_{\varphi \in B} |(f \ast \varphi)(x)|.
$$
Suppose that $B$ is contained and bounded in $\mathcal{S}^{M_p,k}_{A_p,\lambda}(\R^d)$ for some $k,\lambda > 0$. There is $N \in \N$  such that the mapping $\ast_f: \mathcal{S}^{M_p,k}_{A_p,\lambda}(\R^d) \rightarrow Cw^{-1}_N(\R^d)$ is well-defined and continuous. Hence
$$
|V_\psi f(x,\xi)| e^{M(h\xi)}  \leq  \sup_{\varphi \in B}\| f \ast \varphi\|_{w^{-1}_N} w_N(x).
$$
\end{proof}
\begin{theorem}\label{conv-average-Roumieu}
Let $f \in \mathcal{S}'^\ast_\dagger(\R^d)$. The following statements are equivalent:
\begin{itemize}
\item[$(i)$] $f \in \mathcal{B}'^{\ast}_{\mathcal{V}}(\R^d)$.
\item[$(ii)$] $f \ast \varphi \in \mathcal{V}^\circ C(\R^d)$ for all $\varphi \in \mathcal{S}^\ast_\dagger(\R^d)$.
\item[$(iii)$] $f \ast \varphi \in \mathcal{V}^\circ C(\R^d)$ for all $\varphi \in \mathcal{S}^\ast_\dagger(\R^d)$ and the mapping
$$
\ast_f : \mathcal{S}^\ast_\dagger(\R^d) \rightarrow \mathcal{V}^\circ C(\R^d): \varphi \rightarrow f \ast \varphi
$$
is continuous.
\item[$(iv)$] $f \ast \varphi \in \mathcal{B}^\ast_{\mathcal{V}^\circ}(\R^d)$ for all $\varphi \in \mathcal{S}^\ast_\dagger(\R^d)$.
\item[$(v)$] $f \ast \varphi \in \mathcal{B}^\ast_{\mathcal{V}^\circ}(\R^d)$ for all $\varphi \in \mathcal{S}^\ast_\dagger(\R^d)$ and the mapping
$$
\ast_f : \mathcal{S}^\ast_\dagger(\R^d) \rightarrow \mathcal{B}^\ast_{\mathcal{V}^\circ}(\R^d): \varphi \rightarrow f \ast \varphi
$$
is continuous.
\end{itemize}
\end{theorem}
\begin{proof}
The proof is similar to the one of Theorem \ref{conv-average-Beurling} and therefore omitted.
\end{proof}
In the next two corollaries we employ the notation $\ddagger = (B_p)$ or $\{B_p\}$ to treat the Beurling and Roumieu case simultaneously.  However, $\ast$ and $\ddagger$ are always both either of Roumieu or Beurling type, that is, we only consider the classical Gelfand-Shilov spaces \eqref{GS-1} and not the mixed type spaces \eqref{GS-2}. Notice that these corollaries improve results from\footnote{Our spaces  $\mathcal{O}'^{\ast,\ast}_C(\R^d)$ are the same as the convolutor spaces denoted by $\mathcal{O}'^{\ast}_C(\R^d)$  in \cite{D-P-P-V}. We point out however that, due to an error carried from \cite[Prop.\ 2]{D-P-Ve} to the proof of \cite[Thm.\ 3.2]{D-P-P-V} in the Roumieu case, the space $X=\mathcal{O}^{\{M_p\}}_C(\R^d)$  defined on \cite[p.\ 407]{D-P-P-V} is not a predual of $\mathcal{O}'^{\{M_p\}}_C(\R^d)$ since one only has $X'\subsetneq \mathcal{O}'^{\{M_p\}}_C(\R^d)$.} \cite{D-P-P-V,P-P-V} that were obtained under much stronger assumptions and rather different methods (via the parametrix method).
\begin{corollary}
\label{cor conv char 1}
Let $B_p$ be a weight sequence satisfying $(M.1)$ and $(M.2)$ such that $A_p \subset B_p$. For $f \in \mathcal{S}'^{\ast}_{\dagger}(\R^d)$ the following statements are equivalent:
\begin{itemize}
\item[$(i)$] $f  \in \mathcal{S}'^{\ast}_{\ddagger}(\R^d)$.
\item[$(ii)$] $f \ast \varphi \in \mathcal{V}_{(B_p)} C(\R^d)$ ($f \ast \varphi \in \mathcal{W}_{\{B_p\}}C(\R^d)$) for all $\varphi \in \mathcal{S}^{\ast}_{\dagger}(\R^d)$.
\item[$(iii)$] $f \ast \varphi \in \mathcal{O}^{\ast,\ddagger}_C(\R^d)$ for all $\varphi \in \mathcal{S}^{\ast}_{\dagger}(\R^d)$.
\end{itemize}
\end{corollary}
\begin{corollary}
\label{cor conv char 2}
Let $B_p$ be a weight sequence satisfying $(M.1)$ and $(M.2)$ such that $A_p \subset B_p$. For $f \in \mathcal{S}'^{\ast}_{\dagger}(\R^d)$ the following statements are equivalent:
\begin{itemize}
\item[$(i)$] $f  \in \mathcal{O}_C'^{\ast,\ddagger}(\R^d)$.
\item[$(ii)$] $f \ast \varphi \in \mathcal{W}_{(B_p)} C(\R^d)$ ($f \ast \varphi \in \mathcal{V}_{\{B_p\}}C(\R^d)$) for all $\varphi \in \mathcal{S}^{\ast}_{\dagger}(\R^d)$.
\item[$(iii)$] $f \ast \varphi \in \mathcal{S}^{\ast}_{\ddagger}(\R^d)$ for all $\varphi \in \mathcal{S}^{\ast}_{\dagger}(\R^d)$.
\end{itemize}
\end{corollary}
\begin{remark}
\label{rk conv char} If $M_p$ is non-quasianalytic, it is  clear that one may replace ``for all $\varphi\in \mathcal{S}^{\ast}_{\dagger}(\mathbb{R}^{d})$'' by ``for all $\varphi\in \mathcal{D}^{\ast}(\mathbb{R}^{d})$'' in all statements from Theorem \ref{conv-average-Beurling}, Theorem \ref{conv-average-Roumieu}, Corollary \ref{cor conv char 1}, and Corollary \ref{cor conv char 2}. 
\end{remark}
\subsection{Topological properties} \label{subsection topological properties} For an $(LF)$-space $E = \varinjlim E_n$ we define 
$$
\mathfrak{S} = \{ B \subset E \, : \, B \mbox{ is contained and bounded in $E_n$ for some $n \in \N$} \}.
$$
We write $bs(E',E)$ for the $\mathfrak{S}$-topology on $E'$ (the topology of uniform convergence on sets of $\mathfrak{S}$). Grothendieck's factorization theorem implies that $bs(E',E)$ does not depend on the defining inductive spectrum of $E$. The first goal of this subsection is to show the ensuing two theorems.
\begin{theorem}\label{topology-Beurling}
The following three topologies coincide on $\mathcal{B}'^{\ast}_{\mathcal{W}}(\R^d)$:
\begin{itemize}
\item[$(i)$] The initial topology with respect to the mapping
$$
\mathcal{B}'^{\ast}_{\mathcal{W}}(\R^d) \rightarrow L_b(\mathcal{S}^\ast_\dagger(\R^d), \mathcal{B}^\ast_{\mathcal{W}^\circ}(\R^d)): f \rightarrow \ast_f.
$$
\item[$(ii)$] The initial topology with respect to the mapping
$$
\mathcal{B}'^{\ast}_{\mathcal{W}}(\R^d) \rightarrow L_b(\mathcal{S}^\ast_\dagger(\R^d), \mathcal{W}^\circ C(\R^d)): f \rightarrow \ast_f.
$$
\item[$(iii)$] $b(\mathcal{B}'^{(M_p)}_{\mathcal{W}}(\R^d),\mathcal{B}^{(M_p)}_{\mathcal{W}}(\R^d))$ ($bs(\mathcal{B}'^{\{M_p\}}_{\mathcal{W}}(\R^d),\mathcal{B}^{\{M_p\}}_{\mathcal{W}}(\R^d))$).
\end{itemize}
\end{theorem}
\begin{proof}
The first topology is clearly finer than the second one. In order to prove that the second topology is finer than the third one, we need to show that for every bounded set $B \subset \mathcal{B}^{(M_p)}_{\mathcal{W}}(\R^d)$ (for every $h> 0$ and every bounded set $B \subset \mathcal{B}^{M_p,h}_{\mathcal{W}}(\R^d)$) there is a bounded set $B' \subset \mathcal{S}^{(M_p)}_{(A_p)}(\R^d)$ ($B' \subset \mathcal{S}^{\{M_p\}}_{\{A_p\}}(\R^d)$), $v \in \overline{V}(\mathcal{W}^{\circ})$, and $C > 0$ such that
$$
\sup_{\varphi \in B} |\langle f, \varphi \rangle| \leq C \sup_{\varphi \in B'} \| f \ast \varphi \|_{v}
$$
for all $f \in \mathcal{B}'^{(M_p)}_{\mathcal{W}}(\R^d)$ ($f \in \mathcal{B}'^{\{M_p\}}_{\mathcal{W}}(\R^d)$). Choose $\psi, \gamma \in \mathcal{S}^{(M_p)}_{(A_p)}(\R^d)$ such that 
$(\gamma, \psi)_{L^2} = 1$. We first treat the Beurling case. Lemma \ref{STFT-test} implies that  
$$
\sup_{\varphi \in B} \sup_{(x, \xi) \in \R^{2d}}|V_{\overline{\gamma}} \varphi(x,\xi)| w_N(x)e^{M(h\xi)} < \infty
$$
for all $N \in \N$ and $h> 0$, which, by Lemma \ref{proj-desc-2} and Lemma \ref{projectivedescriptionlemma}, is equivalent to 
$$
\sup_{\varphi \in B} \sup_{(x, \xi) \in \R^{2d}}\frac{|V_{\overline{\gamma}} \varphi(x,\xi)|e^{M^{h_j}(\xi)}}{v(x)} < \infty
$$
for some $v \in \overline{V}(\mathcal{W}^\circ)$ and $h_j \in \mathcal{R}$. By Lemma \ref{proj-desc-3} there is $\overline{v} \in \overline{V}(\mathcal{W}^\circ)$ such that $
v(x)/\overline{v}(x) = O(|x|^{-(d+1)})$. Hence Corollary \ref{STFT-reg-dual-GS} and Lemma \ref{STFT-conv} yield that
\begin{align*}
\sup_{\varphi \in B} |\langle f, \varphi \rangle| &\leq \sup_{\varphi \in B}\int \int_{\R^{2d}}|V_\psi f(x, \xi)| |V_{\overline{\gamma}}\varphi(x, - \xi)| \dx \dxi \\
&\leq C\sup_{(x,\xi) \in \R^{2d}} |V_\psi f(x, \xi)|\overline{v}(x)e^{-M^{h_j/H}(\xi)} \\
&\leq C\sup_{\varphi \in B'} \| f \ast \varphi \|_{\overline{v}},
\end{align*}
for some bounded set $B' \subset \mathcal{S}^{(M_p)}_{(A_p)}(\R^d)$. Next, we consider the Roumieu case. Set $k = \pi h/\sqrt{d}$. Lemma \ref{STFT-test} implies that  
$$
\sup_{\varphi \in B} \sup_{(x, \xi) \in \R^{2d}}|V_{\overline{\gamma}} \varphi(x,\xi)| w_N(x)e^{M(k\xi)} < \infty
$$
for all $N \in \N$, which, by Lemma \ref{proj-desc-2}, is equivalent to
$$
\sup_{\varphi \in B} \sup_{(x, \xi) \in \R^{2d}}\frac{|V_{\overline{\gamma}} \varphi(x,\xi)|e^{M(k\xi)}}{v(x)} < \infty
$$
for some $v \in \overline{V}(\mathcal{W}^\circ)$. Using Lemma \ref{proj-desc-3}, there is $\overline{v} \in \overline{V}(\mathcal{W}^\circ)$ such that $v(x)/\overline{v}(x) = O(|x|^{-(d+1)})$. Hence, in view of Corollary \ref{STFT-reg-dual-GS} and Lemma \ref{STFT-conv},
\begin{align*}
\sup_{\varphi \in B} |\langle f, \varphi \rangle| &\leq \sup_{\varphi \in B}\int \int_{\R^{2d}}|V_\psi f(x, \xi)| |V_{\overline{\gamma}}\varphi(x, - \xi)| \dx \dxi \\
&\leq C\sup_{(x,\xi) \in \R^{2d}} |V_\psi f(x, \xi)|\overline{v}(x)e^{-M(k\xi/H)} \\
&\leq C\sup_{\varphi \in B'} \| f \ast \varphi \|_{\overline{v}},
\end{align*}
for some bounded set $B' \subset \mathcal{S}^{\{M_p\}}_{\{A_p\}}(\R^d)$.

Finally, we show that the third topology is finer than the first one. We start with the Beurling case. Since the strong dual of $\mathcal{B}^{(M_p)}_{\mathcal{W}}(\R^d)$ is bornological (cf.\ Proposition \ref{LFS-1}), it suffices to show that every strongly bounded set $B$ in $\mathcal{B}'^{(M_p)}_{\mathcal{W}}(\R^d)$ is also bounded for the first topology. Let $N \in \N$ and $C,h > 0$ be such that
$$
|\langle f, \varphi \rangle | \leq C\|\varphi \|_{\mathcal{B}^{M_p,h}_{w_N}}, \qquad \varphi \in \mathcal{B}^{(M_p)}_{\mathcal{W}}(\R^d),
$$
for all $f \in B$. There is $M > N$ such that
$$
w_N(x+t) \leq C' w_M(x) e^{A(\lambda t)}, \qquad x,t \in \R^d,
$$ 
for some $C', \lambda > 0$. Hence Lemma \ref{translation-norm} implies that for every bounded set $B' \subset \mathcal{S}^{(M_p)}_{(A_p)}(\R^d)$ and every $k > 0$ it holds that
\begin{align*}
\sup_{f \in B} \sup_{\varphi \in B'} \|f \ast \varphi\|_{B^{M_p,k}_{w_M}}  &\leq C\sup_{\varphi \in B'} \sup_{\alpha \in \N^d} \sup_{x \in \R^d} \frac{k^{|\alpha|}\|T_{x}(\check{\varphi}^{(\alpha)})\|_{\mathcal{B}^{M_p,h}_{w_N}}}{M_\alpha w_M(x)} \\
&\leq CC' \sup_{\varphi \in B'} \sup_{\alpha \in \N^d} \frac{k^{|\alpha|}\|\varphi^{(\alpha)}\|_{\mathcal{S}^{M_p,h}_{A_p, \lambda}}}{M_\alpha} < \infty,
\end{align*}
whence $B$ is bounded for the first topology. Next, we consider the Roumieu case. By Proposition \ref{projectivedescription} it suffices to show that for every bounded set $B \subset \mathcal{S}^{\{M_p\}}_{\{A_p\}}(\R^d)$, every $v \in \overline{V}(\mathcal{W}^\circ)$ and every $h_j \in \mathcal{R}$, there are $h > 0$ and a bounded set $B' \subset \mathcal{B}^{M_p,h}_{\mathcal{W}}(\R^d)$ such that
$$
\sup_{\varphi \in B} \|f \ast \varphi \|_{\mathcal{B}^{M_p,h_j}_{v}} \leq \sup_{\varphi \in B'} |\langle f, \varphi \rangle|,
$$
for all $f \in \mathcal{B}'^{\{M_p\}}_{\mathcal{W}}(\R^d)$. Let $k, \lambda > 0$ be such that $B$ is contained and bounded in $\mathcal{S}^{M_p,k}_{A_p, \lambda}(\R^d)$. Lemma \ref{translation-norm} implies that 
$$
B' = \left\{ \frac{T_{x}(\check{\varphi}^{(\alpha)})v(x)}{M_\alpha \prod_{j = 0}^{|\alpha|}h_j } \, : \, \varphi \in B, \alpha \in \N^d, x \in \R^d \right\}
$$
is bounded in $\mathcal{B}^{M_p,k/H}_{\mathcal{W}}(\R^d)$. Hence
\begin{align*}
\sup_{\varphi \in B} \|f \ast \varphi \|_{\mathcal{B}^{M_p,h_j}_{v}} &\leq \sup_{\varphi \in B} \sup_{\alpha \in \N^d} \sup_{x \in \R^d} \frac{|\langle f, T_{x}(\check{\varphi}^{(\alpha)}) \rangle| v(x)}{M_\alpha \prod_{j = 0}^{|\alpha|}h_j} \\
&\leq \sup_{\varphi \in B'} |\langle f, \varphi \rangle|.
\end{align*}
\end{proof} 
\begin{theorem}\label{topology-Roumieu}
The following three topologies coincide on $\mathcal{B}'^{\ast}_{\mathcal{V}}(\R^d)$:
\begin{itemize}
\item[$(i)$] The initial topology with respect to the mapping
$$
\mathcal{B}'^{\ast}_{\mathcal{V}}(\R^d) \rightarrow L_b(\mathcal{S}^\ast_\dagger(\R^d), \mathcal{B}^\ast_{\mathcal{V}^\circ}(\R^d)): f \rightarrow \ast_f.
$$
\item[$(ii)$] The initial topology with respect to the mapping
$$
\mathcal{B}'^{\ast}_{\mathcal{V}}(\R^d) \rightarrow L_b(\mathcal{S}^\ast_\dagger(\R^d), \mathcal{V}^\circ C(\R^d)): f \rightarrow \ast_f.
$$
\item[$(iii)$] $bs(\mathcal{B}'^{(M_p)}_{\mathcal{V}}(\R^d),\mathcal{B}^{(M_p)}_{\mathcal{V}}(\R^d))$ ($b(\mathcal{B}'^{\{M_p\}}_{\mathcal{V}}(\R^d),\mathcal{B}^{\{M_p\}}_{\mathcal{V}}(\R^d))$).
\end{itemize}
\end{theorem}
\begin{proof}
The proof is similar (but simpler) to the one of Theorem \ref{topology-Beurling} and therefore omitted.
\end{proof}
We endow $\mathcal{B}'^{\{M_p\}}_{\mathcal{W}}(\R^d)$ and $\mathcal{B}'^{(M_p)}_{\mathcal{V}}(\R^d)$ with one of the three identical topologies considered in Theorem \ref{topology-Beurling} and Theorem \ref{topology-Roumieu}, respectively.

\begin{lemma}\label{density-S-in-dual}
The embedding 
\begin{equation}
\iota: \mathcal{S}^\ast_\dagger(\R^d) \rightarrow \mathcal{B}'^\ast_{\mathcal{W}}(\R^d): \varphi \rightarrow \left(\psi \rightarrow \int_{\R^d} \varphi(x)\psi(x)\dx \right)
\label{density-S-def}
\end{equation}
has dense range.
\end{lemma}
\begin{proof}
By Corollary \ref{dense-inclusion} it suffices to show that
$$
\iota: \mathcal{B}^\ast_{\mathcal{W}^\circ}(\R^d) \rightarrow \mathcal{B}'^\ast_{\mathcal{W}}(\R^d): \varphi \rightarrow \left(\psi \rightarrow \int_{\R^d} \varphi(x)\psi(x)\dx \right)
$$
has dense range. Choose $\chi \in \mathcal{S}^{(M_p)}_{(A_p)}(\R^d)$ with $\int_{\R^d}\chi(x) \dx = 1$ and set $\chi_n = n^d\chi(n \cdot)$, $n \geq 1$. Fix $f \in \mathcal{B}'^\ast_{\mathcal{W}}(\R^d)$, by Theorem \ref{conv-average-Beurling} we have that $f \ast \chi_n \in \mathcal{B}^\ast_{\mathcal{W}^\circ}(\R^d)$. We claim that $\iota(f \ast \chi_n) \rightarrow f$ in $\mathcal{B}'^\ast_{\mathcal{W}}(\R^d)$, or, equivalently,  that
$ \ast_{\iota(f \ast \chi_n)} \rightarrow \ast_f$ in  $L_b(\mathcal{S}^\ast_\dagger(\R^d), \mathcal{B}^\ast_{\mathcal{W}^\circ}(\R^d))$. Since $\ast_{\iota(f \ast \chi_n)} = \ast_f \circ \ast_{\chi_n}$, where
$ \ast_{\chi_n} \in L(\mathcal{S}^\ast_\dagger(\R^d),\mathcal{S}^\ast_\dagger(\R^d))$ is defined via $\ast_{\chi_n}(\varphi) = \chi_n \ast \varphi$, $\varphi \in \mathcal{S}^\ast_\dagger(\R^d)$, the claim follows from the fact that $\ast_{\chi_n} \rightarrow \operatorname{id}$ in $L_b(\mathcal{S}^\ast_\dagger(\R^d),\mathcal{S}^\ast_\dagger(\R^d))$. 
\end{proof}
\begin{lemma}
The embedding 
$$
\iota: \mathcal{S}^\ast_\dagger(\R^d) \rightarrow \mathcal{B}'^\ast_{\mathcal{V}}(\R^d): \varphi \rightarrow \left(\psi \rightarrow \int_{\R^d} \varphi(x)\psi(x)\dx \right)
$$
has dense range.
\end{lemma}
\begin{proof}
This can be shown in the same way as Lemma \ref{density-S-in-dual}.
\end{proof}
The ensuing two theorems may be considered as quantified versions of Grothendieck's theorem \cite[Chap.\ II, Thm.\ 16, p.\ 131]{Grothendieck} in the setting of tempered ultradistributions.
\begin{theorem}\label{Roumieu-dual}
$\mathcal{B}'^{\{M_p\}}_{\mathcal{W}}(\R^d)$ is a $(PLS)$-space\footnote{A l.c.s.\ is said to be a $(PLS)$-space if it can be written as the projective limit of a projective spectrum consisting of $(DFS)$-spaces. We refer to the survey article \cite{Domanski-artykul} for more information on $(PLS)$-spaces.
} whose strong dual is canonically isomorphic to $\mathcal{B}^{\{M_p\}}_{\mathcal{W}}(\R^d)$, i.e., the embedding
\begin{equation}
\mathcal{B}^{\{M_p\}}_{\mathcal{W}}(\R^d) \rightarrow (\mathcal{B}'^{\{M_p\}}_{\mathcal{W}}(\R^d))'_b: \varphi \rightarrow (f \rightarrow \langle f, \varphi \rangle)
\label{canonical-embedding}
\end{equation}
is a topological isomorphism.
Moreover, consider the following conditions:
\begin{itemize}
\item[$(i)$] $\mathcal{W}$ satisfies $(DN)$.
\item[$(ii)$] $\mathcal{B}^{\{M_p\}}_\mathcal{W}(\R^d)$ satisfies one of the equivalent conditions $(ii)$--$(v)$ from Theorem \ref{reg-Beurling}.
\item[$(iii)$] $\mathcal{B}'^{\{M_p\}}_{\mathcal{W}}(\R^d)$ is ultrabornological.
\item[$(iv)$] $\mathcal{B}'^{\{M_p\}}_{\mathcal{W}}(\R^d)$ is equal to the strong dual of $\mathcal{B}^{\{M_p\}}_{\mathcal{W}}(\R^d)$.
\end{itemize}
Then, $(i) \Rightarrow (ii)  \Rightarrow (iii)  \Rightarrow (iv) \Rightarrow (ii)$. Furthermore, if $M_p$ satisfies $(M.1)$, $(M.2)$, and $(M.3)$, then $(ii) \Rightarrow (i)$. 
\end{theorem}
\begin{proof}
We first show that $\mathcal{B}'^{\{M_p\}}_{\mathcal{W}}(\R^d)$ is a $(PLS)$-space. $L_b(\mathcal{S}^{\{M_p\}}_{\{A_p\}}(\R^d), \mathcal{B}^{\{M_p\}}_{\mathcal{W}^\circ}(\R^d))$ is a $(PLS)$-space because of Proposition \ref{LFS-1} and the fact that $L_b(E,F) \cong L_b(F'_b,E'_b)$ is a $(PLS)$-space for general $(DFS)$-spaces $E$ and $F$ \cite[Prop.\ 4.3]{D-L}. Since a closed subspace of a $(PLS)$-space is again a $(PLS)$-space, it therefore suffices to show that the embedding $\mathcal{B}'^{\{M_p\}}_{\mathcal{W}}(\R^d) \rightarrow  L_b(\mathcal{S}^{\{M_p\}}_{\{A_p\}}(\R^d), \mathcal{B}^{\{M_p\}}_{\mathcal{W}^\circ}(\R^d)): f \rightarrow \ast_f$ has closed range. Let $(f_i)_i \subset \mathcal{B}'^{\{M_p\}}_{\mathcal{W}}(\R^d)$ be a net such that $\ast_{f_i} \rightarrow S$ in $L_b(\mathcal{S}^{\{M_p\}}_{\{A_p\}}(\R^d), \mathcal{B}^{\{M_p\}}_{\mathcal{W}^\circ}(\R^d))$. Define
$$
\langle f, \varphi \rangle :=  S(\check{\varphi})(0), \qquad \varphi \in \mathcal{S}^{\{M_p\}}_{\{A_p\}}(\R^d). 
$$
Clearly, $f \in \mathcal{S}'^{\{M_p\}}_{\{A_p\}}(\R^d)$. Moreover, for each $\varphi \in \mathcal{S}^{\{M_p\}}_{\{A_p\}}(\R^d)$ we have that
\begin{align*}
(f\ast \varphi)(x) &= \langle f, T_x \check{\varphi} \rangle = S(T_{-x}\varphi) (0) = \lim_{i} (f_i \ast T_{-x} \varphi)(0) \\
&= (T_{-x} \lim_{i} (f_i \ast \varphi))(0) = (T_{-x} S(\varphi))(0) = S(\varphi)(x).
\end{align*}
This shows that $S = \ast_f$ and, by Theorem \ref{conv-average-Beurling}, $f \in \mathcal{B}'^{\{M_p\}}_{\mathcal{W}}(\R^d)$. 

Next, we show that \eqref{canonical-embedding} is a topological isomorphism. Clearly, the $bs$-topology is coarser than the strong topology and finer than the weak-$\ast$ topology on $\mathcal{B}'^{\{M_p\}}_{\mathcal{W}}(\R^d)$. The fact that $\mathcal{B}^{\{M_p\}}_{\mathcal{W}}(\R^d)$ is barreled (as it is an $(LF)$-space) therefore implies that a subset of $\mathcal{B}'^{\{M_p\}}_{\mathcal{W}}(\R^d)$ is equicontinuous if and only if it is $bs$-bounded, which in turn yields that \eqref{canonical-embedding} is a strict morphism. We now show that it is surjective. Let $\Phi \in  (\mathcal{B}'^{\{M_p\}}_{\mathcal{W}}(\R^d))'$ be arbitrary and set $f = \Phi \circ \iota \in \mathcal{S}'^{\{M_p\}}_{\{A_p\}}(\R^d)$, where $\iota$ is defined via \eqref{density-S-def}.  We claim that $f \in \mathcal{B}^{\{M_p\}}_\mathcal{W}(\R^d)$, i.e., that there is $\chi \in \mathcal{B}^{\{M_p\}}_\mathcal{W}(\R^d)$ such that $\langle f , \varphi \rangle = \int_{\R^d}\chi(x)\varphi(x) \dx$ for all $\varphi \in \mathcal{S}^{\{M_p\}}_{\{A_p\}}(\R^d)$. Hence $\Phi(\iota(\varphi)) = \int_{\R^d}\chi(x)\varphi(x) \dx = \langle \iota(\varphi), \chi \rangle$ for all $\varphi \in \mathcal{S}^{\{M_p\}}_{\{A_p\}}(\R^d)$ and the result would follow from the fact that $\iota( \mathcal{S}^{\{M_p\}}_{\{A_p\}}(\R^d))$ is dense in $\mathcal{B}'^{\{M_p\}}_{\mathcal{W}}(\R^d)$ (Lemma \ref{density-S-in-dual}). We now prove the claim with the aid of Proposition \ref{STFT-test-char}. Let $\psi \in \mathcal{S}^{(M_p)}_{(A_p)}(\R^d) \backslash \{0\}$. Since $\Phi$ is continuous, there are $h > 0$ and a bounded set $B \subset \mathcal{B}^{M_p,h}_\mathcal{W}(\R^d)$ such that
\begin{align*}
|V_\psi f(x,\xi)| = |\langle f, \overline{M_\xi T_x\psi}\rangle| &= |\Phi(\iota( \overline{M_\xi T_x\psi}))| \leq \sup_{\varphi \in B} |\langle \iota( \overline{M_\xi T_x\psi}), \varphi \rangle| \\
&= \sup_{\varphi \in B} \left | \int_{\R^d} \varphi(t)  \overline{M_\xi T_x\psi(t)} \dt \right | = \sup_{\varphi \in B} |V_\psi \varphi(x,\xi)|.
\end{align*}
Hence, the required bounds for $|V_\psi f|$ directly follow from Lemma \ref{STFT-test}. 

We now turn to the chain of implications. $(i) \Rightarrow (ii)$ and $(ii) \Rightarrow (i)$ (under the extra assumption that $M_p$ satisfies $(M.1)$, $(M.2)$, and $(M.3)$) have already been shown in Theorem \ref{reg-Beurling}. 

$(ii) \Rightarrow (iii)$  Since for general $(LF)$-spaces $E$ it holds that $bs(E',E) = b(E',E)$ if $E$ is regular, we obtain that $\mathcal{B}'^{\{M_p\}}_{\mathcal{W}}(\R^d)$ is equal to the strong dual of $\mathcal{B}^{\{M_p\}}_{\mathcal{W}}(\R^d)$. Hence $(iii)$ follows from Proposition \ref{LFS-1} and the fact that the strong dual of a complete Schwartz space is ultrabornological \cite[p.\ 43]{Schwartz-57}. 

$(iii) \Rightarrow (iv)$ The strong topology on $\mathcal{B}'^{\{M_p\}}_{\mathcal{W}}(\R^d)$ is clearly finer than the $bs$-topology. As the strong dual of an $(LF)$-space is strictly webbed \cite[Prop.\ IV.3.3]{DeWilde}, they are identical by De Wilde's open mapping theorem. 

$(iv) \Rightarrow (ii)$ Since the mapping \eqref{canonical-embedding} is a topological isomorphism, $\mathcal{B}^{\{M_p\}}_\mathcal{W}(\R^d)$ is reflexive and, thus, quasi-complete. 
\end{proof}
\begin{theorem}\label{Beurling-dual}
$\mathcal{B}'^{(M_p)}_{\mathcal{V}}(\R^d)$ is a $(PLS)$-space whose strong dual is canonically isomorphic to $\mathcal{B}^{(M_p)}_{\mathcal{V}}(\R^d)$. Moreover, consider the following conditions:
\begin{itemize}
\item[$(i)$] $\mathcal{V}$ satisfies $(\Omega)$.
\item[$(ii)$] $\mathcal{B}^{(M_p)}_\mathcal{V}(\R^d)$ satisfies one of the equivalent conditions $(ii)$-$(v)$ from Theorem \ref{reg-Roumieu-1}.
\item[$(iii)$]$\mathcal{B}'^{(M_p)}_\mathcal{V}(\R^d)$ is ultrabornological.
\item[$(iv)$] $\mathcal{B}'^{(M_p)}_{\mathcal{V}}(\R^d)$ is equal to the strong dual of $\mathcal{B}^{(M_p)}_{\mathcal{V}}(\R^d)$.
\end{itemize}
Then, $(i) \Rightarrow (ii)  \Rightarrow (iii)  \Rightarrow (iv) \Rightarrow (ii)$. Furthermore, if $M_p$ satisfies $(M.1)$, $(M.2)$, and $(M.3)$, then $(ii) \Rightarrow (i)$. 
\end{theorem}
\begin{proof}
The proof is similar to the one of Theorem \ref{Roumieu-dual} and therefore omitted.
\end{proof}

The next two corollaries settle answers to the question posed after \cite[Thm.\ 3.3, p. 413]{D-P-P-V}.

\begin{corollary}\label{topology-OCD}
Let $M_p$ and $A_p$ be weight sequences satisfying our standard assumptions \eqref{group-cond}. Then, $\mathcal{O}_C'^{(M_p),(A_p)}(\R^d)$ is ultrabornological if $A_p$ satisfies $(M.2)^\ast$. If, in addition, $M_p$ satisfies $(M.1)$, $(M.2)$, and $(M.3)$, then $\mathcal{O}_C'^{(M_p),(A_p)}(\R^d)$ is ultrabornological if and only if $A_p$ satisfies $(M.2)^\ast$. 
\end{corollary}

Naturally, in view of  Theorem \ref{topology-Roumieu} and Theorem \ref{Beurling-dual}, the space $\mathcal{O}_C'^{(M_p),(A_p)}(\R^d)$ is ultrabornological if and only if the strong topology on $\mathcal{O}_C'^{(M_p),(A_p)}(\R^d)$ coincides with the initial topology with respect to the mapping 
$$
\mathcal{O}_C'^{(M_p),(A_p)}(\R^d)\to L_b(\mathcal{S}^{(M_p)}_{(A_p)}(\R^d),\mathcal{S}^{(M_p)}_{(A_p)}(\R^d)): \ f\to \ast_{f}.
$$
 On the other hand, by Remark \ref{rk incomplete}, the situation in the Roumieu case is completely different from that of Corollary 
\ref{topology-OCD}.

\begin{corollary}
\label{topology-OCD-Roumieu}
Let $M_p$ be a weight sequence satisfying $(M.1)$, $(M.2)$, and $(M.3)$, and let $A_p$ be a weight sequence satisfying $(M.1)$ and $(M.2)$. Then, $\mathcal{O}_C'^{\{M_p\},\{A_p\}}(\R^d)$ is not ultrabornological; in particular,  the strong topology on $\mathcal{O}_C'^{\{M_p\},\{A_p\}}(\R^d)$ is strictly finer than the initial topology with respect to the mapping 
$$\mathcal{O}_C'^{\{M_p\},\{A_p\}}(\R^d)\to L_b(\mathcal{S}^{\{M_p\}}_{\{A_p\}}(\R^d),\mathcal{S}^{\{M_p\}}_{\{A_p\}}(\R^d)): \ f\to \ast_{f}.$$
\end{corollary}

\end{document}